\documentclass{amsart}
\usepackage{graphicx} 
\usepackage{amsmath,amstext,amssymb,bm}
\usepackage{}
\numberwithin{equation}{section}
\usepackage{xcolor}
\usepackage{soul}
\usepackage{tikz}
\usetikzlibrary{shapes,arrows}
\usepackage{mathrsfs}
\usepackage{stmaryrd}
\usepackage{tensor}

\newtheorem{definition}{Definition}[section]
\newtheorem{proposition}{Proposition}[section]
\newtheorem{remark}{Remark}[section]
\newtheorem{theorem}{Theorem}[section]
\newtheorem{corollary}{Corollary}[section]
\newtheorem{lemma}{{Lemma}}[section]

\allowdisplaybreaks[4]

\newcommand{\eps}{\varepsilon}
\newcommand{\Ome}{{\Omega}}

\newcommand{\p}{{\partial}}

\newcommand{\R}{\mathbb{R}}

\newcommand{\lDa}{\tensor[^{-}]{\mathcal{D}}{^{\alpha}}}
\newcommand{\rDa}{\tensor[^{+}]{\mathcal{D}}{^{\alpha}}}
\newcommand{\zDa}{\tensor[^{z}]{\mathcal{D}}{^{\alpha}}}
\newcommand{\lrDa}{\tensor[^{\pm}]{\mathcal{D}}{^{\alpha}}}
\newcommand{\rlDa}{\tensor[^{\mp}]{\mathcal{D}}{^{\alpha}}}
\newcommand{\lDb}{\tensor[^{-}]{\mathcal{D}}{^{\beta}}}

\newcommand{\lrDb}{\tensor[^{\pm}]{\mathcal{D}}{^{\beta}}}
\newcommand{\rlDb}{\tensor[^{\mp}]{\mathcal{D}}{^{\beta}}}
\newcommand{\lNa}{\tensor[^{-}]{\mathcal{N}}{^{\alpha}}}
\newcommand{\rNa}{\tensor[^{+}]{\mathcal{N}}{^{\alpha}}}

\newcommand{\lrNa}{\tensor[^{\pm}]{\mathcal{N}}{^{\alpha}}}

\newcommand{\lIa}{\tensor[^{-}]{I}{^{\alpha}}}
\newcommand{\rIa}{\tensor[^{+}]{I}{^{\alpha}}}
\newcommand{\lrIa}{\tensor[^{\pm}]{I}{^{\alpha}}}
\newcommand{\lIoa}{\tensor[^{-}]{I}{^{1-\alpha}}}
\newcommand{\rIoa}{\tensor[^{+}]{I}{^{1-\alpha}}}
\newcommand{\lrIoa}{\tensor[^{\pm}]{I}{^{1-\alpha}}}
\newcommand{\rlIoa}{\tensor[^{\mp}]{I}{^{1-\alpha}}}

\newcommand{\lIob}{\tensor[^{-}]{I}{^{1-\beta}}}

\newcommand{\lLa}{\tensor[^{-}]{\Delta}{^{\alpha}}}
\newcommand{\rLa}{\tensor[^{+}]{\Delta}{^{\alpha}}}
\newcommand{\lrLa}{\tensor[^{\pm}]{\Delta}{^{\alpha}}}
\newcommand{\zLa}{\tensor[^{z}]{\Delta}{^{\alpha}}}
\newcommand{\lWap}{\tensor[^{-}]{W}{^{\alpha,p}}}
\newcommand{\rWap}{\tensor[^{+}]{W}{^{\alpha,p}}}
\newcommand{\lrWap}{\tensor[^{\pm}]{W}{^{\alpha,p}}}
 
\newcommand{\zWap}{\tensor[^{z}]{W}{^{\alpha,p}}}

\newcommand{\lcWap}{\tensor[^{-}]{\mathring{W}}{^{\alpha,p}}}
\newcommand{\rcWap}{\tensor[^{+}]{\mathring{W}}{^{\alpha,p}}}
\newcommand{\lrcWap}{\tensor[^{\pm}]{\mathring{W}}{^{\alpha,p}}}
\newcommand{\lHa}{\tensor[^{-}]{H}{^{\alpha}}}
\newcommand{\rHa}{\tensor[^{+}]{H}{^{\alpha}}}
\newcommand{\lrHa}{\tensor[^{\pm}]{H}{^{\alpha}}}
\newcommand{\lrcHa}{\tensor[^{\pm}]{\mathring{H}}{^{\alpha}}}

\newcommand{\lrHta}{\tensor[^{\pm}]{H}{^{2\alpha}}}
\newcommand{\zHa}{\tensor[^{z}]{H}{^{\alpha}}}
\newcommand{\lT}{\tensor[^{-}]{T}{}}
\newcommand{\rT}{\tensor[^{+}]{T}{}}
\newcommand{\lrT}{\tensor[^{\pm}]{T}{}}
\newcommand{\rlT}{\tensor[^{\mp}]{T}{}}
\begin{document}

\title[Fractional Calculus of Variations]{On A New Class of Fractional Calculus of Variations and Related Fractional Differential Equations
}

\author{Xiaobing Feng\dag}
\thanks{\dag Department of Mathematics, The University of Tennessee, Knoxville, TN 37996, U.S.A. (xfeng@math.utk.edu). The work of this author was partially supported by the NSF grants: DMS-1620168 and DMS-2012414.}

\author{Mitchell Sutton\ddag}
\thanks{\ddag Department of Mathematics, The University of Tennessee, Knoxville, TN 37996, U.S.A. (msutto11@vols.utk.edu). The work of this author was partially supported by the NSF grants DMS-1620168 and DMS-2012414.}


 
\begin{abstract}
	This paper is concerned with analyzing a class of fractional calculus of variations problems
	and their associated Euler-Lagrange (fractional differential) equations. Unlike the existing 
	fractional calculus of variations which is based on the classical notion of fractional 
	derivatives, the fractional calculus of variations considered in this paper is based on 
	a newly developed notion of weak fractional derivatives and their associated fractional order Sobolev spaces. Since fractional 
	derivatives are direction-dependent, using one-sided fractional derivatives and their 
	combinations leads to new types of calculus of variations and fractional differential 
	equations as well as nonstandard Neumann boundary operators. The primary objective of
	this paper is to establish the well-posedness and regularities for a class of fractional calculus of variations problems and their Euler-Lagrange (fractional differential) equations.  
	This is achieved first for one-sided Dirichlet energy functionals which lead to
	one-sided fractional Laplace equations, then for more general energy functionals 
	which give rise to more general fractional differential equations.

\end{abstract}



\maketitle

\tableofcontents

 

\section{Introduction}\label{sec-1}

Let $V$ be a Banach space of real-valued functions defined on a bounded domain 
$\Omega \subset \mathbb{R}^d (d\geq 1)$ and $E:V\to \mathbb{R}$ be a functional 
defined on $V$.  The calculus of variations
concerns with the following minimization problem: 
\begin{align}\label{eq1.1}
u= \underset{v\in V}{\mbox{\rm argmin}}\,  E(v).
\end{align}
A typical energy functional $E$  has the following integral form: 
\begin{align}\label{eq1.2}
  E(v)=\int_\Omega f(D v, v, x)\, dx,
\end{align}
where $f:\mathbb{R}^d\times \mathbb{R}\times \mathbb{R}^d\to \mathbb{R}$, called the energy 
density function, must depend on the gradient $D v$ (or part of it). The dependence 
of $f$ on higher order derivatives of $v$ is also allowed.
Such calculus of variations problems arise from many scientific and engineering fields
such as differential geometry, physics, mechanics, materials sciences, and image processing,
just to name a few. The calculus of variations has been a well-developed field in mathematics 
(cf. \cite{Giaquinta,Dacorogna,Evans} and the references therein). 

Recent advances in fractional/nonlocal calculus and differential equations 
\cite{Podlubny,Samko,Feng_Sutton1} as well as their applications \cite{Hilfer,Du2019,Karniadakis} 
have motivated the consideration of fractional calculus of variations  \cite{Torres,Torres1}, 
which conceptually amounts to replacing the integer order gradient $D v$ by a fractional order
gradient $D^\alpha v\, (0<\alpha <1)$ in \eqref{eq1.2}, leading to the following 
prototypical fractional calculus of variations problem:
\begin{align}\label{eq1.3}
u= \underset{v\in V^\alpha}{\mbox{\rm argmin}}\,  E^\alpha(v), 
\end{align}
where $V^\alpha$ stands for some fractional order (Banach) space and 
\begin{align}\label{eq1.4}
E^\alpha(v)=\int_\Omega f(D^\alpha v, v, x)\, dx.
\end{align}

Although the above conceptual extension is easy to achieve, there are some fundamental issues 
and difficulties 
which must be addressed and overcome.  The utmost issue is the meaning/choice of the fractional 
gradient/derivative $D^\alpha v$ in \eqref{eq1.4}, because there are multiple 
definitions of fractional derivatives (which may not be equivalent) used in the literature.
We recall that the well-known classical fractional derivative concepts include 
Riemann-Liouville, Caputo, Fourier, and Gr\"unwald-Letnikov fractional order derivatives
(cf. \cite{Samko,Stinga,Feng_Sutton1}). The second main issue, which is also a technical obstruction,
is the compatibility between these classical fractional derivatives $D^\alpha v$  
and the (energy) space $V^\alpha$ in \eqref{eq1.3}. For example,  in the case of Caputo derivative 
$\tensor[^{C}]{D}{^{\alpha}} $,
it requires that $v \in C^1$ (or at least $AC$, which could be relaxed to $H^1$)  to ensure its existence. Consequently, one must have $C^1\subset V^\alpha$ (or $H^1\subset V^\alpha$), which forces one to consider calculus of variations with the following integer-fractional mixed order energy functional (cf. \cite{Torres}):
\begin{align}\label{eq1.5}
J^\alpha(v)=\int_\Omega \varphi(Dv, D^\alpha v, v, x)\, dx
\end{align}
over the stronger space $C^1$ (or $H^1$) and the dependence of $\varphi$ on $Dv$ 
is required.  Finally, another important issue is whether to consider the dependence of 
all one-sided fractional derivatives in the density function $f$ (and $\varphi$) because fractional 
derivatives are often direction-dependent and perhaps in only one direction.

Motivated by the above considerations and issues, in this paper we consider and study fractional order calculus of variations in one spatial dimension given by 
\begin{align}\label{FundamentalMin}
    u = \underset{v \in {_{*}}{\mathcal{W}}^{\alpha,p}_{\theta,\lambda} }{\mbox{argmin}}\, \mathcal{E}^{\alpha}_{p,\theta,\lambda} (v),
\end{align}
where 
\begin{align}\label{eq1.7}
    \mathcal{E}^{\alpha}_{p,\theta,\lambda} (v) : = \int_{\Omega}
     L_{p,\theta,\lambda} \bigl(\lDa v(x) , \rDa v(x) , v(x) , x \bigr) \,dx.
\end{align}
Here $\Omega=(a,b)$, ${_{*}}{\mathcal{W}}^{\alpha, p}_{\theta,\lambda}$ denotes a two-parameter
(i.e. $(\theta, \lambda)$) family of fractional order Sobolev spaces, and $*$ will take value $0$ or 
empty (see Section \ref{sec-2} for the 
details).  We first note that the energy density function $L_{p,\theta,\lambda}$ depends independently on both the left and right fractional derivative $\lDa v$ and $\rDa v$, which allows various combinations of them in the density function. We then note that these two fractional derivatives are not the classical fractional derivatives, instead, they are \textit{weak fractional derivatives} which were introduced and developed recently by the authors in \cite{Feng_Sutton1, Feng_Sutton2} (see Appendix A). They are the natural extensions of the integer order weak derivatives used to define Sobolev spaces $W^{k,p}$ and the foundation of the fractional calculus of variations theory to be presented subsequently as the primary goal of the paper.

The remainder of this paper is organized as follows. Section \ref{sec-2} introduces some space notations and necessary preliminaries to be used in the later sections. The reader is also referred to Appendix A and B for the definitions of weak fractional derivatives and the associated fractional Sobolev spaces and to \cite{Feng_Sutton1,Feng_Sutton2} for their comprehensive analyses. 
Section \ref{sec-3} considers some special density functions $L_{p,\theta,\lambda}$ which give rise  the one-sided fractional $p$-Laplace equations. The focuses of this section are on characterizing one-sided fractional harmonic functions and  deriving the nonstandard fractional Neumann boundary operators via considered variational problems. 
Section \ref{sec-4} deals with the general energy density function $L_{p,\theta,\lambda}$ and establishes the existence of solutions to a class of problems \eqref{FundamentalMin} via the direct method of the calculus of variations. 
Section \ref{sec-5} addresses the existence and uniqueness of solutions to \eqref{FundamentalMin}
in the case $p=2$ via Galerkin formulations and the Lax-Milgram Theorem,  which are important for developing efficient numerical methods \cite{Feng_Sutton3}. Special attention is given to studying  the subtle boundary value problems for   one-sided $2\alpha$-order fractional differential equations. 
Moreover, some regularity results for the weak solutions  are also established. Finally, the paper is concluded with some remarks in Section \ref{sec-6}.


\section{Preliminaries}\label{sec-2}

Throughout the paper we assume $\Omega=(a,b)$ is a finite interval, unless stated otherwise,  
$0 < \alpha < 1$, $1 < p < \infty$, $0 \leq \theta \leq 1$, $\lambda= 0$ or $1$, and let $\Gamma(z)$  
denotes the Euler-Gamma function.  For a given Banach space $V$,  $V^{*}$ denotes its dual space
(the space of bounded linear functionals on $V$). We also note that Appendix A and B contain 
the definitions and properties of weak fractional derivatives and accompanying fractional
Sobolev space theory which were developed in \cite{Feng_Sutton1} and \cite{Feng_Sutton2}. 
We also adopt the function and operator notations used there. 
For instance,  $\lIa$ and $\rIa$ denote the left and right Riemann-Liouville fractional 
integral operators of order $\alpha$ (cf. \cite{Samko, Feng_Sutton1}), and  $\lDa$, $\rDa$, and $\zDa$ are the \textit{left, right and Riesz weak fractional derivatives}, respectfully (cf. Definition \ref{WeakDerivative}). The notation
$\lrDa$ stands for either $\lDa$ or $\rDa$. The functions, $\kappa^{\alpha}_{\pm} : \Omega \rightarrow \R$ stand for the kernel functions of $\lrDa$ (i.e. $\lrDa \kappa^{\alpha}_{\pm} \equiv 0$ in $\Omega$) and $\kappa^{\alpha}_{z}:\Omega \rightarrow \R$ denotes any linear combination of the functions $\kappa^{\alpha}_{z_1}:\Omega \rightarrow \R$ and $\kappa^{\alpha}_{z_2}:\Omega \rightarrow \R$; the two unique elements of the nullspace of the Riesz derivative,  $\mathcal{N}({^{z}}{\mathcal{D}}{^{\alpha}})$ (cf. Proposition \ref{NullRiesz}).

The function spaces, $\lWap(\Omega)$, $\rWap(\Omega)$, $W^{\alpha,p}(\Omega)$, and $\zWap(\Omega)$ denote respectfully the \textit{left, right, symmetric, and Riesz fractional Sobolev spaces} (cf. Definition \ref{Sobolev}). Moreover, $\lcWap(\Omega)$ and $\rcWap(\Omega)$ denote respectively the subspaces
of $\lWap(\Omega)$ and $\rWap(\Omega)$ with $c_\mp^{1-\alpha}=0$ (see Appendix A for the precise  definitions). 
In the case that $p=2$, we use the conventional notation $\lHa(\Omega)$, $\rHa(\Omega)$, $H^{\alpha}(\Omega)$, and $\zHa(\Omega)$ to denote the corresponding Hilbert spaces.

In order to consider a general class of fractional calculus of variation problems and 
to present them in a unified fashion, for $0\leq \theta\leq 1$ and $\lambda=0$ or $1$, we introduce the following family of function spaces:
\begin{align}\label{general_spaces}
 \mathcal{W}^{\alpha,p}_{\theta,\lambda}
 &:=\theta (1-\theta)\, {W}{^{\alpha,p}}(\Omega) + \lambda\Bigl\{  \llbracket\theta \rrbracket   \, \rWap(\Omega)  
 + \llbracket 1-\theta \rrbracket \, \lWap(\Omega) \Bigr\}  \\
 &\qquad 
 + (1-\lambda) \Bigl\{  \llbracket\theta \rrbracket\, \rcWap(\Omega) 
 +   \llbracket 1-\theta \rrbracket \,  \lcWap(\Omega)   \Bigr\}, \nonumber
\end{align}
where $\llbracket \theta \rrbracket$ denotes the integer part of $\theta$. 
It is easy to check that
\begin{align}\label{general_spaces_2}
\mathcal{W}^{\alpha,p}_{\theta, \lambda} = 
\begin{cases}
\lcWap(\Omega) &\text{if } \theta=0 \mbox{ and } \lambda = 0,\\
W^{\alpha,p}(\Omega) &\text{if } 0<\theta<1 \mbox{ and } \lambda=0, \\ 
\lWap(\Omega) &\text{if }\theta = 0 \mbox{ and } \lambda=1,\\
\rcWap(\Omega) &\text{if }\theta = 1 \mbox{ and }\lambda = 0,\\
W^{\alpha,p}(\Omega) &\text{if } 0<\theta<1 \mbox{ and } \lambda=1, \\ 
\rWap(\Omega) &\text{if }\theta =1 \mbox{ and } \lambda =1.
\end{cases}
\end{align}

The norm on $\mathcal{W}^{\alpha,p}_{\theta,\lambda}$ is naturally defined as 
\begin{align}
\|u\|_{\mathcal{W}^{\alpha, p}_{\theta,\lambda}} := 
\begin{cases}
\|u\|_{\lWap(\Omega)} &\text{if } \theta = 0 \mbox{ and } \lambda=0,1, \\
\|u\|_{W^{\alpha,p}(\Omega)} &\text{if } 0<\theta<1 \mbox{ and } \lambda=0,1, \\ 
\|u\|_{\rWap(\Omega)} &\text{if } \theta = 1 \mbox{ and } \lambda=0,1.
\end{cases}
\end{align}

We also introduce, in the case $\alpha p >1$,
\begin{align}\label{zero_trace_spaces}
{_{0}}\mathcal{{W}}^{\alpha, p}_{\theta,\lambda} := \Bigl\{ u \in \mathcal{W}^{\alpha,p}_{\theta,\lambda } (\Omega) \, : \, (1-\theta)\, \lT u = 0 \text{ and } \theta\, \rT u = 0  \Bigr\}.
\end{align}
Here $\lrT$ denotes the trace operators (cf. Definition \ref{trace} for their 
precise meanings).  

\begin{remark}\label{BoundaryNotation}
When $\theta = 1$,  $(1 - \theta)\, \lT$ should not be considered and we only consider the condition ${^{+}}{T} u = 0$. Similarly, when $\theta =0$, $\theta\, \rT$ should not be considered and we only have $\lT u =0$. Finally, if $0 < \theta <1$, then we consider these conditions at both ends of the domain; $\lT u = \rT u = 0$. 
\end{remark}

Additional necessary results related to weak fractional derivatives and fractional Sobolev spaces can be found in  
Appendix A and B, respectfully.


\section{One-sided Fractional Laplace and Neumann Boundary Operators}\label{sec-3}
Because of the dependence of the energy density function $L_{p,\lambda,\theta}$ in \eqref{eq1.7} on one-sided fractional derivatives, $\lDa v$ and $\rDa v$,  which does not have
counterparts in the integer order case, many more scenarios must be considered in the 
fractional calculus of variations. To better understand the new problems and to ease
the presentation and explanation, we first consider some simpler energies of the fractional 
calculus of variations. 
In particular, we shall focus on the case $p=2$, derive/define one-sided fractional Laplace operators and one-sided fractional Neumann boundary operators, and explore basic properties
of these operators.   
In Section \ref{sec-4} and \ref{sec-5}, we shall consider more general energies and among other issues in the fractional calculus of variations, the existence and uniqueness of minimizers.

\begin{definition}
      Let $\alpha >0$.  The functional 
        \begin{align}\label{LRDirichletEnergy}
            {^\pm}{E}^\alpha_p (u) := \dfrac{1}{p} \int_{\Omega} \bigl|\lrDa u \bigr|^{p} \,dx.
        \end{align}
is called the $\alpha$-order left/right Dirichlet $p$-energy, and the functional
        \begin{align}\label{SymmetricDiricheltEnergy}
            E^{\alpha}_{p} (u) : = \frac{1}{2}\Bigl( {^{-}}{E}{^{\alpha}_{p}} (u) + {^{+}}{E}{^{\alpha}_{p}}(u) \Bigr)
        \end{align}
is called the $\alpha$-order symmetric Dirichlet  $p$-energy. Moreover, the functional 
        \begin{align}\label{RieszDirichletEnergy}
            {^{z}}{E}{^{\alpha}_{p}} (u) : = \dfrac{1}{p} \int_{\Omega} \bigl| \zDa u\bigr|^{p} \,dx 
        \end{align}
        is called the $\alpha$-order Riesz $p$-energy. 
\end{definition}

\begin{remark}
   (i) The left/right $\alpha$-order Dirichlet $p$-energy is a special class of energies for which 
   $\theta \in \{ 0 ,1\}$ and $\lambda =0$ in the density function $L_{p,\theta,\lambda}$ so that 
  \begin{align*}
     L_{p,\theta,\lambda} ({^{-}}{\mathcal{D}}{^{\alpha}} v, {^{+}}{\mathcal{D}}{^{\alpha}} v , v ,x)
     = L({^{\pm}}{\mathcal{D}}{^{\alpha}} v).
    \end{align*} 
 Similarly, the  $\alpha$-order Riesz $p$-energy has $\theta = \frac12$ and $\lambda =0$ so that 
    \begin{align*}
         L_{p,\theta,\lambda} ({^{-}}{\mathcal{D}}{^{\alpha}} v, {^{+}}{\mathcal{D}}{^{\alpha}} v , v ,x) 
         =  L ({^{z}}{\mathcal{D}}{^{\alpha}} v).
    \end{align*}
    
    (ii) The left/right $\alpha$-order Dirichlet $p$-energy given by (\ref{LRDirichletEnergy}), is well defined for any $u \in {^{\pm}}{W}{^{\alpha,p}}(\Omega)$, the $\alpha$-order symmetric Dirichlet  $p$-energy given by (\ref{SymmetricDiricheltEnergy}) is well defined for any function $u \in W^{\alpha,p}(\Omega)$, and similarly, the  $\alpha$-order Riesz $p$-energy given by (\ref{RieszDirichletEnergy}) is well defined for any $u \in {^{z}}{W}{^{\alpha,p}}(\Omega)$.
    
\end{remark}

\subsection{One-sided Fractional Laplace Operators}\label{sec-4.1}

A plethora of work has been done in recent years to define and study numerous, sometimes nonequivalent, definitions of fractional Laplace operators. Unlike the existing definitions, the notions to be presented below are based on and related to the notion of weak fractional derivatives and particular energy functionals. 
 This is in concert with the way one may derive the integer Laplacian via \textit{Dirichlet's principle}. The goal of this subsection is to establish this connection methodically.

\begin{proposition}\label{prop3.1}
	Let $u \in \lrWap(\Omega)$ be a minimizer of ${^{\pm}}{E}{^{\alpha}_{p}}$. Then it must satisfy the following fractional differential equation in the distributional sense:
	\begin{align}\label{L/RFractionalLaplaceEq}
	{^{\pm}}{\Delta}^{\alpha}_{p} u := \rlDa \Bigl(\bigl|\lrDa u \bigr|^{p-2} \lrDa u \Bigr)   = 0. 
	\end{align}
\end{proposition}

\begin{proof} 
	The proof follows immediately from the Fundamental Lemma of the Calculus of Variations 
	(\cite{Giaquinta}), which says that the first variation of ${^{\pm}}{E}{^{\alpha}_{p}}$ must vanish 	at $u$.  For completeness, we briefly carry out the derivation below. 
	
	Let $\varphi \in C^{\infty}_{0}(\Omega)$ and define $\Phi: \R\rightarrow \R$ by 
	\begin{align*}
	\Phi(t) := {^{\pm}}{E}{^{\alpha}_{p}} (u+t\varphi).
	\end{align*}
    Since $\Phi$ is a power function in $t$, it is differentiable and 
	\begin{align}\label{step1}
	\Phi'(t) = \int_{\Omega} \bigl|\lrDa (u+t\varphi) \bigr|^{p-2}\, \lrDa(u+t\varphi) \lrDa \varphi \,dx.
	\end{align}
	Since $u$ is a minimizer of ${^{\pm}}{E}{^{\alpha}_{p}}$, then $t=0$ is an extreme point for $\Phi$, hence, it must hold that $\Phi'(0)=0$. Setting $t =0$ in \eqref{step1}, we get 
	\begin{align*}
	\int_{\Omega}  \bigl|\lrDa u \bigr|^{p-2}\, \lrDa u\cdot \lrDa \varphi \,dx = 0 \qquad \forall \, \varphi \in C^{\infty}_{0}(\Omega),
	\end{align*}
	which implies that \eqref{L/RFractionalLaplaceEq} holds in the distributional sense 
	by the weak fractional derivative definition (cf. Definition \ref{WeakDerivative}).
\end{proof}

Similarly, we also can prove the following conclusions. 

\begin{proposition}\label{prop3.2}
	Let $u \in W^{\alpha,p}(\Omega)$ be a minimizer of ${E}{^{\alpha}_{p}}$. Then it must satisfy the following fractional differential equation in the distributional sense:
	\begin{align}\label{SympLaplacian}
	{\Delta}{^{\alpha}_{p}}u := \frac12 \Bigl(\rDa  \bigl|\lDa u \bigr|^{p-2} \lDa u 	
	+  \lDa  \bigl|\rDa u \bigr|^{p-2} \rDa u \Bigr)=0.
	\end{align}
	Moreover, if $u \in \zWap(\Omega)$ is a minimizer of ${^{z}}{E}{^{\alpha}_{p}}$, then it must satisfy the following fractional differential equation in the distributional sense:
	\begin{align}\label{SympLaplacian}
	{^{z}}{\Delta}{^{\alpha}_{p}} u := \zDa  \bigl|\zDa u \bigr|^{p-2} \zDa u=0 .
	\end{align}
\end{proposition}

As the notations suggest, the following definitions are in order.

\begin{definition}\label{LaplacianDef}
        ${^{\pm}}{\Delta}{^{\alpha}_{p}}, {\Delta}{^{\alpha}_{p}}$ and ${^{z}}{\Delta}{^{\alpha}_{p}}$ are called respectively the left/right, symmetric, and Riesz $\alpha$-order fractional 
        $p$-Laplace operators.  When $p=2$, we write $\lrLa := {^{\pm}}{\Delta}{^{\alpha}_{2}}$, 
        $ {\Delta}{^{\alpha}_{}} : =\Delta{^{\alpha}_{2}}$, and ${^{z}}{\Delta}{^{\alpha}}: = {^{z}}{\Delta}{^{\alpha}_{2}}$.
\end{definition}

\begin{remark} 
	(i) Trivially, ${\Delta}{^{\alpha}_{p}}  = \frac{1}{2}( {^{-}}{\Delta}{^{\alpha}_{p}} + {^{+}}{\Delta}{^{\alpha}_{p}})$.   
	
    (ii) We defined above the \textit{fractional $p$-Laplacians} using a weak fractional derivative notion.
     This is analogous to the integer-order Laplacian, $\Delta$. 
     By considering these operators in a weak sense, we avoid the ambiguity of choosing a particular classical fractional derivative over others. For example, defining a fractional Laplacian using the Riemann-Louiville derivatives verses the Caputo derivatives. After eliminating ambiguity of choosing  differential operators, there is no need to worry about 
     what function spaces a strong/classical fractional Laplacian may operate. In the integer case, 
     it is well known that $C^2(\Omega)$ (or $H^2(\Ome)$) is a well-suited space in which to study the action of the Laplacian. However, in the fractional case, the choice of function space depends heavily on the choice of the fractional derivative notion and even then, may not be well understood.
    
    (iii) As is the case for the integer order $p$-Laplacian, each of the above fractional
    $p$-Laplace operators are derived from an $\alpha$-order $p$-energy subordinate to the appropriate fractional derivative notion.
    
\end{remark}

It is easy to see that in the case $p=2$, $\lrLa = \rlDa \lrDa$ takes derivatives in each direction. It may be natural to expect that a fractional Laplacian ought be defined using two derivatives in one direction. However, we like to point out that $\lrLa \neq \lrDa \lrDa$. Why is this the case? The following subsection is dedicated to answering this question and establishing some connections between these two differing $2\alpha$-order differential operators.

\subsection{Properties of Fractional Laplace Operators}\label{sec-4.2}

This subsection is devoted to studying the mapping properties of the fractional Laplace operators defined in Section \ref{sec-4.1}. In particular, we characterize their nullspaces and investigate under what conditions the $2\alpha$-order differentiation in one direction (as opposed to the mixed directions of $\lrLa$) is guaranteed to exist.

\subsubsection{\bf $\alpha$-Harmonic Functions}\label{sec-4.2.1}

\begin{definition}
A function $u\in L^1(\Ome)$ is said to be left/right $\alpha$-harmonic if  $\lrLa u = 0$ in the distributional sense. A function $u\in L^1(\Omega)$
is said to be $\alpha$-harmonic (resp. Riesz $\alpha$-harmonic) if $\Delta^{\alpha} u=0$
(resp. $\zLa u = 0$) in the distributional sense.  
\end{definition}

It comes as no surprise that the kernel space of the left/right fractional Laplacian is directly related to the kernel spaces of the left and right weak fractional derivatives and their mapping properties. Analogously, we recall that the kernel space of the 1-D integer Laplacian consists of constant and linear functions. 

\begin{theorem}
$u$ is left/right $\alpha$-harmonic if and only if $u = c_1 \kappa^{\alpha}_{\pm} + c_2 \lrIa  \kappa_{\mp}^{\alpha}$, 
where $\kappa_{-}^{\alpha}(x) = (x-a)^{\alpha-1}$, $\kappa_{+}^{\alpha}(x) = (b-x)^{\alpha -1}$,
and $\lrIa$ denotes the left/right $\alpha$-order Riemann-Liouville fractional integral operator
(See Appendix A and \cite{Samko, Feng_Sutton1}). 
\end{theorem}

\begin{proof}
The sufficiency is a direct calculation and a consequence of the Fundamental Theorem 
of Weak Fractional Calculus (FTwFC) (cf. Theorem \ref{FTWFC}).

To show the necessity, assume that $u$ is left/right $\alpha$-harmonic and let $\mathcal{N}(\lrDa)$ denote the null/kernel space of the operator $\lrDa$. By assumption $\rlDa \lrDa u =0$. It follows that $\lrDa u \in \mathcal{N} (\rlDa )$. Hence $\lrDa u = c_2 \kappa^{\alpha}_{\mp}$. Applying the left/right $\alpha$-order fractional integral operator $\lrIa$ and by the FTwFC (cf. Theorem \ref{FTWFC}), we have that $u = c_1 \kappa_{\pm}^{\alpha} + c_2 {^{\pm}}{I}{^{\alpha}} \kappa_{\mp}^{\alpha} $. This concludes the proof.
\end{proof}

Next, we prove an analogous result for the symmetric $\alpha$-order fractional Laplacian.

\begin{lemma} \label{harmoniclemma}
$\Delta^{\alpha} \psi = 0$ cannot hold for every $\psi \in C^{\infty}_{0}(\Omega)$. 
\end{lemma}

\begin{proof}
By contradiction, assume $\Delta^{\alpha} \psi = 0$ for every $\psi \in C^{\infty}_{0}(\Omega)$. It follows by the definition of weak fractional derivatives (cf. Definition \ref{WeakDerivative})
\begin{align*}
	&\Delta^{\alpha} \psi = 0\\
	&\Leftrightarrow \quad \int_{\Omega} \lDa \rDa \psi  \,\varphi \,dx = - \int_{\Omega} \rDa \lDa \psi\, \varphi \,dx \quad \forall \, \psi, \varphi \in C^{\infty}_{0}(\Omega)\\ 
	& \Leftrightarrow \quad \int_{\Omega} \rDa \psi \rDa \varphi \,dx = - \int_{\Omega}  \lDa \psi \lDa\varphi \,dx \quad \forall \, \psi, \varphi \in C^{\infty}_{0}(\Omega)\\
	& \Rightarrow \quad \int_{\Omega} (\rDa \psi)^2 \,dx = - \int_{\Omega} (\lDa \psi)  ^{2} \,dx \quad \forall\, \psi \in C^{\infty}_{0}(\Omega).
\end{align*}
This is a contradiction and concludes the proof. 
\end{proof}

\begin{theorem} \label{thm3.2}
$u$ is $\alpha$-harmonic if and only if $u =0$. 
\end{theorem}

\begin{proof}
    The sufficiency  is trivial. Therefore, we only need to prove the necessity.

    Assume that $u$ is $\alpha$-harmonic, that is,  $\rDa \lDa u + \lDa \rDa u = 0$. By the definition and integration by parts for weak fractional derivatives (cf. Definition \ref{WeakDerivative}), it follows that 
    \begin{align*}
    	 & \int_{\Omega} \lDa \rDa u\, \varphi \,dx = - \int_{\Omega} \rDa \lDa u \, \varphi \,dx \quad \forall \, \varphi \in C^{\infty}_{0}(\Omega) \\ 
	 &\Leftrightarrow \quad \int_{\Omega} u \lDa \rDa \varphi \,dx = - \int_{\Omega} u \rDa \lDa \varphi \,dx \quad \forall \, \varphi \in C^{\infty}_{0}(\Omega) \\ 
	 &\Leftrightarrow \quad \int_{\Omega} u \, \Delta^{\alpha} \varphi \,dx  =0\quad \forall \, \varphi \in C^{\infty}_{0}(\Omega).
    \end{align*}
    Applying Lemma \ref{harmoniclemma}, we conclude the proof.
 \end{proof}

We now turn our attention to the Riesz fractional Laplacian. Unlike the previous characterizations, the presence of both cross and same directional differentiation 
in the Riesz fractional Laplacian definition results in a more complicated set 
of harmonic functions.   On the other hand, we have a nice characterization (cf. Proposition \ref{NullRiesz}) of the kernel space $\mathcal{N} (\zDa)$ of $\zDa$,
thanks to \cite[Theorem 4.4]{Cai}. 



A consequence of the characterization 
is the following result which was proved in \cite[Theorem 4.8]{Cai}.

\begin{lemma}\label{HarmonicLemma}
    Let $\Omega = (0,1)$, $\rho(x) = x^{\alpha/2}(1-x)^{\alpha/2}$, and 
    $\kappa_z^\alpha \in \mathcal{N}({^{z}}{\mathcal{D}}{^{\alpha}})$ (cf. Proposition \ref{NullRiesz}). 
    If $2/3 < \alpha <1$, then the equation $\zDa u = \kappa^{\alpha}_{z}$ has the follow general solution

    \begin{align}\label{RieszHarmonic}
        u(x) = \rho(x) \sum_{n=0}^{\infty} u_n G^{(\alpha/2, \alpha,2)}_{n} (x)   + \kappa^{\alpha}_{z}(x), 
    \end{align}
    where 
    \begin{align*}
        u_n &= -\dfrac{\Gamma(n+1)}{\Gamma(n+1+\alpha)} \frac{\int_{0}^{1} \rho(x) \kappa^{\alpha}_{z} (x) G^{(\alpha/2, \alpha/2)}_{n}(x) \,dx}{\| G_{n}^{(\alpha/2 , \alpha/2)}(x) \|_{L^{2}((0,1),\rho)}^{2}}
    \end{align*}
    and
\begin{align*}
    G_{n}^{(\alpha/2,\alpha/2)}(x) := \sum_{k=0}^{n} g_{n,k}^{(\alpha/2, \alpha/2)} x^{k}  
\end{align*}
are the Jacobi polynomials for $n \geq 0$ with 
\begin{align*}
    g_{n,k}^{(\alpha/2 , \alpha/2)} := \dfrac{(-1)^{n+k} \Gamma(n+\alpha/2 +1) \Gamma( n + k +\alpha + 1)}{\Gamma(k+1)\Gamma(n-k+1)\Gamma(n + \alpha + 1) \Gamma(k+\alpha/2 +1)},
\end{align*}
and $L^{2}((0,1) , \rho)$ denotes the weighted $L^2$-space with the weight function $\rho$. 
\end{lemma}

\begin{proof}
It can be shown that when $\alpha >2/3$, $\kappa^{\alpha}_{z} \in L^{2}((0,1) , \rho)$ (cf. Proposition \ref{prop3.3}). It follows by (\cite[Theorem 4.4, 4.8]{Cai}) that $\zDa u = \kappa^{\alpha}_{z}$ has the solution given by (\ref{RieszHarmonic}). 
\end{proof}

\begin{proposition}\label{prop3.3}
    (i) The functions $\kappa^{\alpha}_{z}$ are Riesz $\alpha$-harmonic. (ii) For $2/3 <\alpha < 1$, any Riesz $\alpha$-harmonic function must have the form (\ref{RieszHarmonic}). 
\end{proposition}

\begin{proof}
  (i) Since (by definition) $\zDa\kappa^{\alpha}_{z} =0$, then trivially, 
   $\zDa\zDa \kappa_z =0$. Thus, $\kappa^{\alpha}_{z}$ is Riesz $\alpha$-harmonic.
   
  (ii) If $u$ is given by (\ref{RieszHarmonic}), by Lemma \ref{HarmonicLemma},  we have 
   $\zDa u = \kappa^{\alpha}_{z}$.  Consequently, $\zDa\zDa u= \zDa \kappa^\alpha_z=0$. Thus, $u$ is Riesz $\alpha$-harmonic. Conversely,  if $u$ is Riesz $\alpha$-harmonic, then 
   $\zDa u$ belongs to $\mathcal{N} (\zDa)$, hence, $\zDa u\in \mbox{\rm span} \{ \kappa_{z_1}^\alpha, \kappa_{z_2}^\alpha \}$, then there exist constants $c_1$ and $c_2$ such that $\zDa u=\kappa_z^\alpha:=
   c_1 \kappa_{z_1}^\alpha +c_2 \kappa_{z_2}^\alpha$. It follows by Lemma \ref{HarmonicLemma} that $u$ must be given by
    (\ref{RieszHarmonic}). The proof is complete. 
\end{proof}

\begin{remark}
To the best of our knowledge, it is not known whether there exists a larger class of Riesz $\alpha$-harmonic functions when $0 < \alpha \leq 2/3$ because of lacking an analogue 
of Lemma \ref{HarmonicLemma} in this case. 
\end{remark}

\subsubsection{\bf A Fractional Calder\'on-Zygmund Type Estimate}\label{sec-4.2.2} 
In this subsection we consider how the one-sided fractional Laplace operator may offer control on a pure one-sided second-order derivative. This is in the spirit of the so-called Calder\'on-Zygmund inequality:
\begin{align*}
\| D^2 u\|_{L^{p}(\Omega)} \leq C \left( \|u\|_{L^{p}(\Omega)} + \|\Delta u\|_{L^{p}(\Omega)}\right)
\end{align*}
 where $D^2 u$ is the total second-order derivative of $u$. In the integer one-dimensional case, this estimate is trivial. However, when considering $2 \alpha$-order differentiation, left-right directions in one-dimension is akin to $x$-$y$ in the integer two-dimensional setting. This begs the question of whether single direction differentiation can be controlled by assumptions on the one-sided fractional Laplacian.

\begin{proposition}\label{Directions}
    Let $u\in L^1(\Omega)$, $0 < \beta < \alpha <1$, and $1 < p,q < \infty$. 
    If $\alpha > 1/p$ and $\lrLa u \in L^{p}(\Omega)$, then $ \lrDb \lrDa u \in L^{q}(\Omega)$ 
    for  all $\alpha - \beta > 1/p$ and $\beta < 1/q$.
    \end{proposition}
    
    \begin{proof}
    We only prove the assertion for the left direction with $\Omega = (a,b)$ because the other follows similarly.
    
    Set $v:=\rDa \lDa u= \lLa u \in L^p((a,b))$. By the FTwFC (cf. Theorem \ref{FTWFC}), we have
    \begin{align}\label{1} 
    \lDa u(x) = \rIa v(x)  + \Gamma(1-\alpha)^{-1} [ \rIoa \lDa u](b) (b-x)^{\alpha-1}.
    \end{align}
    Taking the $\beta$-order left fractional derivative on both sides of (\ref{1}) yields
    \begin{align*}
       & \lDb \lDa u(x)\\ 
       &\quad = \lDb \rIa v(x) + \Gamma(1-\alpha)^{-1}[ \rIoa  \lDa u] (b) \lDb (b-x)^{\alpha-1} \\
       &\quad = : J_1 + J_2 .
    \end{align*}
    We now calculate and estimate $J_1$ and $J_2$. Since $v \in L^{p}((a,b))$, it follows from  \cite[Theorem 3.6]{Samko} that $\rIa v \in C^{\alpha - 1/p}([a,b])$ and by 
    \cite[Theorem 3.1]{Samko}) we get
    \begin{align}\label{2}
        \lIob \rIa v(x) = \dfrac{\rIa v(a)}{\Gamma(1+\alpha)} (x-a)^{1 - \beta} + \psi(x),
    \end{align}
    where $\psi \in C^{1 - 1/p + \alpha - \beta}([a,b])$ such that $|\psi(x)| \leq (x-a)^{1- 1/p + \alpha -\beta}.$ Hence $\psi \in C^1([a,b])$ when $\alpha  - \beta > 1/p$. Then 
    \begin{align}\label{3}
        \lDa \rIa v (x) = \frac{\rIa v(a)}{\Gamma(1+\alpha) (1-\beta)^{-1}} (x-a)^{-\beta} + \psi'(x),
    \end{align}
    where $\psi' \in C([a,b])$. Since $\beta < 1/q$, it follows that $J_1= \lDb \rIa v \in L^{q}(\Omega)$. 
    
    To estimate $J_2$, we need to calculate and estimate $\lDb (b-x)^{\alpha-1}.$ First, we show that ${^{RL}_{a}}{D}{^{\beta}_{x}} (b-x)^{\alpha-1} \in L^{1}((a,b))$. Then by the characterization of weak fractional derivatives (cf. \cite{Feng_Sutton1}), $\lDb (b-x)^{\alpha-1}$ coincides with the Riemann-Liouville derivative. 
    The same calculation can easily be altered to show that it belongs to  $L^1_{\mbox{\rm \tiny loc}}(\Omega)$ when $\beta \leq \alpha$. By direct calculation we obtain
   
    \begin{align*}
        \int_{a}^{b} \left| {^{-}}{\mathcal{D}}{^{\beta}} \kappa^{\alpha}_{+}(x)\right|\,dx 
        &= \frac{1}{\Gamma(1 - \beta)}\int_{a}^{b} \left|  \frac{d}{dx} \int_{a}^{x} \frac{(b-y)^{\alpha -1}}{(x-y)^{\beta}}  \,dy\right|\,dx \\ 
        &= \frac{1}{\Gamma(1-\beta)} \int_{a}^{b}\left( \frac{(b-a)^{\alpha -1}}{(x-a)^{\beta}} + (1-\alpha) \int_{a}^{x} \frac{(b-y)^{\alpha-2}}{(x-y)^{\beta}} \,dy \right)\,dx\\
         &= \frac{1}{\Gamma(1-\beta)} \biggl( \frac{(b-a)^{\alpha -\beta}}{1-\beta} + (1-\alpha) \int_{a}^{b} \int_{y}^{b} \dfrac{(b-y)^{\alpha-2}}{(x-y)^{\beta}} \,dxdy \biggr)\\
        &= \frac{1}{\Gamma(1-\beta)} \biggl( \frac{(b-a)^{\alpha -\beta}}{1-\beta} + \frac{1-\alpha}{1-\beta} \int_{a}^{b} (b-y)^{\alpha - \beta -1} \biggr)\\
        &= C(\alpha,\beta) (b-a)^{\alpha -\beta} < \infty.
    \end{align*}
    Then, recall that (cf. \cite{Samko})
    \begin{align*}
        {^{-}}{I}{^{\tau}} \left( (x-a)^{\sigma -1} (b-x)^{\gamma -1}\right) = \frac{\Gamma(\sigma)}{\Gamma(\tau + \sigma)} \frac{(x-a)^{\tau + \sigma -1}}{(b-a)^{1-\gamma}} {_{2}}{F}{_{1}} \left( \sigma, 1- \gamma ; \tau + \sigma , \dfrac{x-a}{b-a}\right) ,
    \end{align*}
    where ${_{2}}{F}{_{1}}(a,b;c,z)$ is the Gauss hypergeometric function defined 
    \begin{align}\label{Hypergeometric}
        {_{2}}{F}{_{1}} (a,b;c,z) : = \frac{\Gamma(c)}{\Gamma(b)\Gamma(c-b)} \int_{0}^{1} t^{b-1} (1-t)^{c-b-1} (1-zt)^{-a}\,dt,
    \end{align}
    if $0 < Re(b) < Re(c)$ and $|\mbox{arg}(1-z)|< \pi$. Since 
    \begin{align*}
        \lDb (b-x)^{\alpha -1} = \dfrac{1}{\Gamma(1-\beta)} \dfrac{(b-a)^{\alpha-1}}{(x-a)^{\beta}} + \dfrac{(1-\alpha)}{\Gamma(1-\beta)} \lIob (b-x)^{\alpha-2},
    \end{align*}
    by (\ref{Hypergeometric}), we have  
    \begin{align*}
        &\lIob (b-x)^{\alpha -2}
        = \dfrac{\Gamma(1)}{\Gamma(2-\beta)} \dfrac{(x-a)^{1-\beta}}{(b-a)^{2-\alpha}} {_{2}}{F}{_{1}} \left(1 , 2-\alpha ; 2-\beta, \dfrac{x-a}{b-a} \right)\\
        &\qquad = \dfrac{(b-a)^{\alpha-2}}{\Gamma(2-\alpha)\Gamma(\alpha -\beta)} (x-a)^{1-\beta} \int_{0}^{1} t^{1-\alpha} (1-t)^{\alpha - \beta -1} \left( 1- \frac{x-a}{b-a} t \right)^{-1} \,dt.
    \end{align*}
    Let $\gamma$ and $\mu$ be H\"older conjugates and see
    \begin{align*}
        & \int_{a}^{b} \left| (x-a)^{1-\beta} \int_{0}^{1} t^{1-\alpha} (1-t)^{\alpha - \beta -1} \left(1 - \frac{x-a}{b-a} t\right)^{-1}\,dt \right|^{q}\,dx\\
        & \leq  \int_{a}^{b} \left|\int_{0}^{1} (1-t)^{\alpha -\beta -1} (x-a)^{1 -\beta} \left( 1- \frac{x-a}{b-a} t\right)^{-1} \,dt  \right|^{q}dx\\
        &\leq \int_{a}^{b} \left( \int_{0}^{1} (1-t)^{\mu(\alpha -\beta -1)}\,dt \right)^{q/\mu} \left| \int_{0}^{1} (x-a)^{(1-\beta)\gamma} \left( 1- \frac{x-a}{b-a} t\right)^{-\gamma} \,dt \right|^{q/\gamma}\,dx \\ 
         &= C \int_{a}^{b} \left| \dfrac{(x-a)^{(1-\beta)\gamma}(a-b)}{(1-\gamma)(x-a)} \left[ \left(1-\frac{x-a}{b-a} \right)^{1-\gamma} -1 \right] \right|^{q/\gamma}\,dx\\
        &= C  \int_{a}^{b} \left| (x-a)^{(1 -\beta) \gamma -1} \left( \left(1- \frac{x-a}{b-a}\right)^{1-\gamma} -1 \right) \right| ^{q/\gamma}\,dx\\
         &= C \int_{a}^{c} \left| (x-a)^{(1-\beta)\gamma -1} \left( \left( 1-\frac{x-a}{b-a} \right)^{1-\gamma} -1 \right) \right|^{q/\gamma} \,dx \\
        &\quad + C \int_{c}^{b} \left| (x-a)^{(1-\beta)\gamma -1} \left(\left(1-\frac{x-a}{b-a} \right)^{1-\gamma} -1 \right) \right|^{q/\gamma}\,dx \\
        &\leq  C \left( \int_{a}^{c} (x-a)^{((1-\beta) \gamma -1)q/\gamma}\,dx +  \int_{c}^{b} \left(1- \frac{x-a}{b-a} \right)^{q(1-\gamma)/\gamma} \,dx\right) < \infty \\
        \end{align*} provided that $((1-\beta)\gamma -1)q/\gamma + 1 > 0$ and $q/\gamma(1-\gamma) + 1 > 0$. It is easy to verify that these conditions are satisfied under the assumptions on $\beta$ and $q$. The proof is complete. 
    \end{proof}

    \begin{corollary}
        Let $u \in L^1(\Omega)$, $0 < \alpha , \beta <1$, $1 \leq p,q<\infty$. If $\alpha > 1/p$ and $\lrDa \lrDa u \in L^{p}(\Omega)$, then $\rlDb \lrDa u \in L^{q}(\Omega)$ for all $\alpha - \beta > 1/p$ and $\beta < 1/q$.
    \end{corollary}

    \begin{proof}
    The result follows by similar calculations and estimates as in the proof of Proposition \ref{Directions}. 
    \end{proof}

\subsection{Fractional Neumann Boundary Operators and Green's Identity}\label{sec-4.3}
It is expected that the Dirichlet boundary conditions for the one-sided Poisson equations to only be given at one endpoint of the domain. This is holistically consistent with the trace concept (cf. Definition \ref{trace}) in the one-sided spaces ${^{\pm}}{W}{^{\alpha,p}}(\Omega)$, whose 
functions may be weakly singular at the other endpoint of the domain. Therefore, 
a Dirichlet boundary condition could not (and should not) be assigned there. 
This is indeed the case as to be seen in Section \ref{sec-4} and \ref{sec-5}. 
Another type of widely used boundary condition for integer order PDEs is the Neumann (or natural) 
boundary condition whose physical meaning is the prescribed normal flux. An interesting 
question is what would be the `right' fractional Neumann (or natural) boundary condition. Since fractional differential operators are nonlocal, it is not clear which fractional operator physically represents flux. In turn, that makes the identification of the fractional Neumann boundary operator a delicate and difficult task. 

The goal of this subsection is to define a fractional Neumann (or natural) boundary operator (and condition) and to show its consistency with the fractional calculus of variations. 

\begin{definition}\label{NeumannDef}
Let $u:\Ome\to \mathbb{R}$. Define the operator, 
\begin{align}\label{Neumann}
    {^{\pm}}{\mathcal{N}}{^{\alpha}_{p}} u : =  \lrT \rlIoa \left| \lrDa u \right|^{p-2} {^{\pm}}{\mathcal{D}}{^{\alpha}} u,
\end{align}
 called the left/right fractional Neumann boundary operator associated with the fractional $p$-Laplacian ${^{\pm}}{\Delta}{^{\alpha}_{p}}$. 
When, $p=2$, we write $\lrNa := {^{\pm}}{\mathcal{N}}{^{\alpha}_{2}}$. 
\end{definition}

\begin{remark}
(i) Specifically, when $p=2$ and $\Omega = (a,b)$, we have
\begin{align*}
\lNa u &= \lT \rIoa \lDa u = \bigl(\rIoa \lDa u \bigr)(b),\\
\rNa u &= \rT \lIoa \rDa u = \bigl(\lIoa \rDa \bigr)(a).
\end{align*}
Similar to the trace concept in the space ${^{\pm}}{H}{^{\alpha}}(\Omega)$, we again see a one-sided concept of the fractional Neumann boundary operator that depends on the direction of differentiation.

(ii) Again, we see a mixing of the directions each operator is taken. For example, in the left case, we take a right fractional integral on top of a left fractional derivative. Moreover, 
unlike the integer order Neumann operator which is defined by the normal derivative at the boundary, 
the fractional version relies on a mixing of two fractional operators. Neumann boundary conditions are referred to as \textit{natural boundary conditions} because they are embedded in and arise as natural consequences of the associated calculus of variations problems. This point of view will be explained with details below.

(iii) A natural question is whether the integration `undoes' the differentiation in (\ref{Neumann}). Since the order of integration does not match that of differentiation, the orders do not `cancel' and we truly have a nonlocal operator that is distinct from 
the trace operator. 
\end{remark}

The above definition of fractional Neumann boundary operators is motivated by the 
following theorem.  

\begin{theorem}\label{DeriveNeumann}
    Suppose that 
    \[
    u = \underset{v \in \lrWap}{\mbox{\rm argmin }} {^{\pm}}{E}{^{\alpha}_{p}} (v),
    \]
    then ${^{\pm}}{\mathcal{N}}{^{\alpha}_{p}} u =0$ in the distributional sense. 
\end{theorem}

\begin{proof}
	The proof is similar to that of Proposition \ref{prop3.1}.  It suffices to assume that $u$ is smooth due to the density property of $\lrWap$-functions. 
	
	Let $v \in C^{\infty}(\overline{\Omega})$. It follows by the minimizer assumption and taking the first variation of ${^{\pm}}{E}{^{\alpha}_{p}}$ that 
    \begin{align*}
        0 &= \int_{\Omega} |\lrDa u|^{p-2} \lrDa u \lrDa v \,dx \\
        &= \int_{\Omega} |\lrDa u |^{p-2} \lrDa u \left( \frac{1}{\Gamma(1-\alpha)} \frac{\rlT v}{\kappa^{1+\alpha}_{\pm}} + \lrIoa v' \right) \,dx \\ 
        &= \int_{\Omega} |\lrDa u |^{p-2} \lrDa u \left( \dfrac{1}{\Gamma(1-\alpha)} \frac{\rlT v}{\kappa^{1+\alpha}_{\pm}}\right) + \rlIoa |\lrDa u|^{p-2} \lrDa u v' \,dx \\ 
        &= \int_{\Omega} |\lrDa u |^{p-2} \lrDa u \left( \frac{1}{\Gamma(1-\alpha)} \frac{\rlT v}{\kappa^{1+\alpha}_{\pm}}\right) + \rlDa |\lrDa u |^{p-2} \lrDa u v\,dx \\ 
        &\quad + \lrT \rlIoa |\lrDa u |^{p-2} \lrDa u  \lrT v - \rlT \rlIoa |\lrDa u |^{p-2} \lrDa u  \lrT v\\
        &= \int_{\Omega} {^{\pm}}{\Delta}{^{\alpha}_{p}} u v\,dx 
        + \lrT \rlIoa |\lrDa u |^{p-2} \lrDa u \lrT v\\
        &= : \int_{\Omega} {^{\pm}}{\Delta}{^{\alpha}_{p}} u v\,dx + {^{\pm}}{\mathcal{N}}{^{\alpha}_{p}} u\, \lrT v.
    \end{align*}
    Here we have used the following identity to obtain the second to last equality  
    \begin{align*}
        \int_{\Omega} |\lrDa u |^{p-2} \lrDa u \left( \frac{1}{\Gamma(1-\alpha)} \frac{\rlT v}{\kappa^{1+\alpha}_{\pm}}\right) \,dx = \rlT \rlIoa |\lrDa u |^{p-2} \lrDa u  \lrT v.
    \end{align*}
    Hence,  ${^{\pm}}{\Delta}{^{\alpha}_{p}} u = 0$ and ${^{\pm}}{\mathcal{N}}{^{\alpha}_{p}} u=0$ in the distributional sense.  The proof is complete.
 
\end{proof}

The above proof also infers the following useful result.

\begin{corollary}\label{cor-3.2}
	There holds the following  fractional Green's identity 
	for the fractional $p$-Laplacian ${^\pm}{\Delta}^\alpha_p$:
	\begin{align}
	\int_{\Omega} |\lrDa u|^{p-2} \lrDa u \lrDa v \,dx 
	= \int_{\Omega} {^\pm}{\Delta}^\alpha_p u \cdot v \,dx + {^\pm}{\mathcal{N}}^\alpha_p u  \lrT v 
	\end{align}
	when $u$ and $v$ are appropriately chosen.
\end{corollary}

\begin{remark}
	(i) The validity of the above Green's 
	identity shows that both the trace operator $\lrT$ and the fractional Neumann 
	operator ${^{\pm}}{\mathcal{N}}{^{\alpha}_{p}}$ are good and natural generalizations of 
	their integer counterparts, for the one-sided fractional operators.
	
	(ii) In the literature (cf. \cite{Wang} and the references therein),  
	${^{\mp}}{T} {^{\pm}}{\mathcal{D}}{^{\alpha}} u$ is also defined as a fractional Neumann boundary condition by mimicking the integer order operator. However, $\rlT \lrDa u = 0$ is not equivalent to ${^{\pm}}{\mathcal{N}}{^{\alpha}_{p}}u = 0$. To see this point, set $u = {^{\pm}}{I}{^{\alpha}} c$ for $c > 0$, then
    \begin{align*}
        {^{\pm}}{\mathcal{N}}{^{\alpha}_{p}} u &: = \lrT \rlIoa |\lrDa u|^{p-2} \lrDa u 
        = \lrT \rlIoa  c^{p-1}
        = c^{p-1} \lrT \kappa^{-\alpha}_{\mp} 
        = 0.
    \end{align*}
    However, it follows from the FTwFC (cf. Theorem \ref{FTWFC}) that $\lrDa u =c$, which implies
     that $\rlT \lrDa u > 0$. Hence, these two conditions are not equivalent. Therefore, 
     defining $\rlT \lrDa$ as a fractional Neumann boundary operator is inconsistent with the 
     embedded natural boundary condition from the fractional calculus of variations. 
\end{remark}
 
Here we see that in the same spirit as the fractional Laplacian, the fractional Neumann boundary operator can be obtained via the calculus of variations arguments. Moreover, prescribing a fractional Neumann boundary condition for a given Euler-Lagrange (fractional differential) equation is equivalent to considering a fractional calculus of variations problem with the natural boundary condition. This equivalence may not be true for other definitions of fractional Neumann boundary operators proposed in the literature. 

\section{Fractional Calculus of Variations via Direct Method}\label{sec-4}

In this section, we consider the general fractional calculus of variations problem 
(\ref{FundamentalMin}). Our goal is to establish the existence of minimizers under some 
structure conditions on the density function $L_{p,\theta,\lambda}$ using the 
direct method (cf. \cite{Dacorogna, Evans}).   


We first take a closer look at the meanings of the three subscripts on $L_{p,\theta,\lambda}$. 
The parameter $p$ is obvious, which is an index inherited from the fractional 
Sobolev space $\lrWap$.  
The parameter $\theta \in [0,1]$ can be thought of as a linear weight (or selector parameter) between the left and right fractional derivatives. That is, $\theta$ controls the symmetry of $L_{p,\theta,\lambda}$ with respect to left and right fractional differentiation. 
The last parameter $\lambda= 0$ or $1$ characterizes the role of a zero order term of $v$.
In particular, $\lambda=0$ indicates that $L_{p,\theta,\lambda}$ does not depend 
on $v$ explicitly. 
Therefore, we may assume that $L_{p,\theta,\lambda}$ has the following form:
\begin{align}\label{eq4.1}
L_{p,\theta,\lambda} (\lDa v, \rDa v,v,x) = L_p\bigl( (1-\theta)\lDa v , \theta \rDa v, \lambda v, x\bigr)
\end{align}
for some function $L_p: \mathbb{R}^3\times \Omega\to \mathbb{R}$. 
It is clear that if $\theta \in (0,1)$, then $L_p$ depends on both $\lDa v$ and $\rDa v$, 
but when $\theta=0$ or $1$, $L_p$ depends only on one of them. This situation leads to
so-called  one-sided $2\alpha$-order fractional differential equations
to which there is no integer order counterparts. 
We also remark that the special case when $\theta =1/2$, $\lambda=1$, and $L_p$ depends on $\lDa v$ and $\rDa v$ indirectly via their arithmetic average $\mathcal{D}^\alpha v$. That is,
\[
L_{p,\theta,\lambda}(\lDa v, \rDa v,v,x) = L_p\bigl(\mathcal{D}^\alpha  v, v, x\bigr).
\]
We shall consider this special case separately in Section \ref{sec-5.5} since it gives rise to a fundamentally different problem.

In the remainder of this section, we shall study the existence of minimizers to (\ref{FundamentalMin}) under some suitable structure conditions on the Lagrangian \eqref{eq4.1}. Before stating such a result, we first prove a sufficient condition for  $L_{p,\theta,\lambda}$ 
to be weak lower semicontinuous; a familiar component from the study of the Calculus 
of Variations (cf. \cite{Evans, Dacorogna}).

\begin{proposition}\label{LowerSemiContinuity}
	Assume that $L_{p,\theta,\lambda}: \R^3\times \Omega \to \R$ is smooth, bounded from below, and convex in its first two arguments. Moreover, there exists two smooth functions $L^1_{p,\theta,\lambda},L^2_{p,\theta,\lambda} : \R^2 \times \Omega \to \R$ such that 
	\begin{align} \label{structure1}
	 \frac{\partial}{\partial a_1}  L_{p,\theta,\lambda}(\lDa v,\rDa v , v,x) &= L^{1}_{p,\theta,\lambda}(\lDa v , v,x) \\ 
	\frac{\partial}{\partial a_2} L_{p,\theta,\lambda}(\lDa v,\rDa v , v,x) &= L^{2}_{p,\theta,\lambda}(\rDa v, v,x), \label{structure2}
	\end{align}
	where $\frac{\p }{\p a_i}L_{p,\theta,\lambda} \, (i=1,2)$ stands for the partial 
	derivative of $L_{p,\theta,\lambda}$ with respect to the $i$th argument. 
	Then the energy functional $\mathcal{E}^\alpha_{p,\theta,\lambda}$ is weakly lower semicontinuous 
	on $\mathcal{W}^{\alpha,p}_{\theta,\lambda}$. 
\end{proposition}

\begin{proof} 
	Let $\{v_k\}_{k=1}^{\infty} \subset \mathcal{W}^{\alpha,p}_{\theta,\lambda}$ and $v_k \rightharpoonup v$ 
	in $\mathcal{W}^{\alpha,p}_{\theta,\lambda}$ and set $$\ell := \underset{k \rightarrow \infty}{\mbox{liminf}}\, \mathcal{E}^\alpha_{p,\theta,\lambda}(v_k).$$
	 We want to show that $\mathcal{E}^\alpha_{p,\theta,\lambda}(v) \leq \ell$. 
	 
	 Since $v_k \rightharpoonup v$, it follows that $\{v_k\}_{k=1}^{\infty}$ is a bounded sequence. Hence, there exists $M >0$ so that $\underset{k} \sup \|v_k\|_{\mathcal{W}^{\alpha,p}_{\theta,\lambda}} \leq M$. Passing to a subsequence, without relabeling, $\ell = \underset{k \rightarrow \infty}{\lim} \mathcal{E}^\alpha_{p,\theta,\lambda}(v_k)$. By a precompactness result (cf. Lemma \ref{Precompact}), $\mathcal{W}^{\alpha,p}_{\theta,\lambda}\subset \subset L^{p}(\Omega)$. It follows that $v_{k} \rightarrow v$ in $L^{p}(\Omega)$. For yet another subsequence, without relabeling, $v_k \rightarrow v$ a.e. in $\Omega$. 
	
	Fix $\eps > 0$. Since $v_k \rightarrow v$ a.e. in $\Omega$, by Egorov's theorem, there exists $\Omega_{\eps} \subset \Omega$ so that $|\Omega \setminus \Omega_{\eps} | < \eps$ and $v_k \rightarrow v$ uniformly on $\Omega_{\eps}$. Assume that $\Omega_{\eps} \subset \Omega_{\eps}' \subset \Omega$ for $0 < \eps ' < \eps$. Define 
	\begin{align*}
	U_{\eps} := \bigl\{ x \in \Omega \, : \, |v| + |{^{-}}{\mathcal{D}}{^{\alpha}} v| +  |{^{+}}{\mathcal{D}}{^{\alpha}} v| \leq 1/\eps \bigr\}.
	\end{align*}
	Then $|\Omega \setminus U_{\eps}| \rightarrow 0$ as $\eps \rightarrow 0$. Set $V_{\eps} := \Omega_{\eps} \cap U_{\eps}$ and note that since $|\Omega \setminus \Omega_{\eps}| < \eps$ and $|\Omega \setminus U_{\eps}| \rightarrow 0$ as $\eps \rightarrow 0$, this implies that $|\Omega \setminus V_{\eps} | \rightarrow 0$ as $\eps \rightarrow 0$.
	
	Recall that $L_{p,\theta,\lambda}$ is bounded from below.
	Without loss of generality, we assume $L_{p,\theta,\lambda} \geq 0$; otherwise we repeat this argument for $L_{p,\theta,\lambda} + C$ for sufficiently large constant $C>0$. It follows from  the convexity of $L_{p,\theta,\lambda}$ that 
	\begin{align*}
	&\mathcal{E}^\alpha_{p,\theta,\lambda}(v_k)
	 = \int_{\Omega} L_{p,\theta,\lambda}(\lDa v_k , \rDa v_k , v_k , x) \,dx \\ 
	&\quad \geq \int_{V_{\eps}} L_{p,\theta,\lambda}(\lDa v_k , \rDa v_k , v_k , x) \,dx \\ 
	&\quad \geq \int_{V_{\eps}} L_{p,\theta,\lambda} (\lDa v , \rDa v_k, v_k ,x)\,dx\\
	&\qquad \quad + \int_{V_{\eps}} \frac{\partial }{\partial a_1} L_{p,\theta,\lambda}(\lDa v, \rDa v_k , v_k , x)   \lDa (v_k - v)\,dx\\
	&\quad \geq \int_{V_{\eps}} L_{p,\theta,\lambda}(\lDa v ,\rDa v , v_k,x) \,dx 
	+ \int_{ V_{\eps} } L^{1}_{p,\theta,\lambda} (\lDa v ,v_{k},x) \lDa (v_k -v)\,dx\\
	&\quad \quad + \int_{ V_{\eps} } \frac{\partial}{\partial a_2} L_{p,\theta,\lambda} (\lDa v , \rDa v, v_k , x) \rDa (v_k - v)\,dx \\
	&\quad = \int_{V_{\eps}} L_{p,\theta,\lambda}(\lDa v ,\rDa v , v_k,x) \,dx 
	 +  \int_{ V_{\eps} } L^{1}_{p,\theta,\lambda} (\lDa v ,v_{k},x)  \lDa (v_k - v)\,dx\\
	&\qquad \quad + \int_{ V_{\eps} } L^2_{p,\theta,\lambda}(\rDa v , v_k ,x)  \rDa (v_k - v)\,dx.
	\end{align*}
	By the uniform convergence on $V_{\eps} \subset \Omega_{\eps}$,
	\begin{align*}
	\lim_{k \rightarrow \infty} \int_{V_{\eps}} L_{p,\theta,\lambda} (\lDa v , \rDa v , v_k ,x)\,dx = \int_{V_{\eps}} L_{p,\theta,\lambda} (\lDa v , \rDa v , v ,x)\,dx.
	\end{align*}
	Moreover, since 
	\begin{align*}
	L^{1}_{p,\theta,\lambda}({^{-}}{\mathcal{D}}{^{\alpha}} v ,v_{k} ,x) \to L^{1}_{p,\theta,\lambda} (\lDa v , v,x ), \\
	L^{2}_{p,\theta,\lambda}(\rDa v ,v_{k} ,x) \to L^{2}_{p,\theta,\lambda} (\rDa v , v,x )
	\end{align*}
	uniformly on $V_{\eps}$ and $\lrDa v_k \rightharpoonup \lrDa v$ in $L^{p}(\Omega)$ we have  
	\begin{align*}
	\lim_{k \rightarrow \infty} \int_{V_{\eps}} L^{1}_{p,\theta,\lambda} (\lDa v , v_k ,x)  (\lDa v_k - \lDa v) \,dx &= 0,\\
	\lim_{k \rightarrow \infty} \int_{V_{\eps}} L^{2}_{p,\theta,\lambda} (\rDa v , v_k ,x)  (\rDa  v_k - \rDa v) \,dx &= 0.
	\end{align*}
	Thus,
	\begin{align*}
	\ell = \lim_{k \rightarrow \infty} \mathcal{E}^\alpha_{p,\theta,\lambda}(v_k) \geq \int_{V_{\eps}} L_{p,\theta,\lambda} (\lDa v , \rDa v , v, x )\,dx \qquad \forall \eps > 0.
	\end{align*}
	Finally, it follows from the monotone convergence theorem that 
	\begin{align*}
	\ell \geq \int_{\Omega} L_{p,\theta,\lambda} (\lDa v , \rDa v , v, x )\,dx = \mathcal{E}^\alpha_{p,\theta,\lambda} (v).
	\end{align*}
	This completes the proof.
\end{proof}

\begin{remark}
	The structure conditions \eqref{structure1} and \eqref{structure2} imply that $L_{p,\theta,\lambda}$
	does not contain product terms of $\lDa v$ and $\rDa v$. 
\end{remark}

To ensure the existence of minimizers, we need the following assumption: there exists $c_0 > 0$ and $c_1 \geq 0$ so that
\begin{align}\label{LagrangeCoercive}
L_{p,\theta,\lambda} (\lDa v ,\rDa v ,v,x) \geq \dfrac{c_0}{2} \left(|\lDa v|^{p}+ |\rDa v|^{p}+ |v|^{p}\right) - c_1.
\end{align}
The above assumption ensures that the energy functional $\mathcal{E}^\alpha_{p,\theta,\lambda}$ satisfies the following \textit{coercive condition}:
\begin{align}\label{EnergyCoercive}
\mathcal{E}^\alpha_{p,\theta,\lambda} (v) \geq c_0  \|v\|_{\mathcal{W}^{\alpha, p}_{\theta,\lambda}(\Omega)}^{p} - c_1 |\Omega|.  
\end{align}

We now are ready to state and prove the desired existence theorem. 

\begin{theorem}\label{Existence}
	Assume that $L_{p,\theta,\lambda}$ satisfies the conditions of Proposition \ref{LowerSemiContinuity} and the coercive condition (\ref{LagrangeCoercive}). 
	Then there exists $u \in {_{0}}\mathcal{{W}}^{\alpha,p}_{\theta,\lambda}$ which solves  problem (\ref{FundamentalMin}).
\end{theorem}

\begin{proof}
	Let $$M := \underset{v \in {_{0}}\mathcal{{W}}^{\alpha,p}_{\theta,\lambda}}{\inf} \, \mathcal{E}^\alpha_{p,\theta,\lambda}(v).$$ 
	Assume $M < \infty$, otherwise, the assertion is trivially true. The coercivity 
	assumption also implies that $M>-c_1|\Omega|$.
	Choose a minimizing sequence $\{v_{k}\}_{k=1}^{\infty} \subset {_{0}}\mathcal{{W}}^{\alpha,p}_{\theta,\lambda}$ such that $\mathcal{E}^\alpha_{p,\theta,\lambda}(v_k) \rightarrow M$ as $k \rightarrow \infty$. 
	It follows from the coercivity assumption (\ref{LagrangeCoercive}) that  
	\begin{align*}
	\mathcal{E}^\alpha_{p,\theta,\lambda}(v) \geq \dfrac{c_0}{2}  \int_{\Omega} \Bigl( |\lDa v|^{p} + |\rDa v|^{p} + |v|^{p} \Bigr)\,dx  - c_1|\Omega|  \qquad \forall\, v \in {_{0}}\mathcal{{W}}^{\alpha,p}_{\theta,\lambda}. 
	\end{align*}
	Since $M < \infty$ we have that $\|v_k\|_{\mathcal{{W}}^{\alpha,p}_{\theta,\lambda}} < \infty$ for every $k$.
	\begin{align*}
	\int_{\Omega} \Bigl( |\lDa v_k |^{p} + |\rDa v_{k} |^{p} + |v_{k}|^{p} \Bigr) \,dx < \infty
	\qquad \forall k\geq 1. 
	\end{align*}
	 Thus, $\{v_k\}_{k=1}^{\infty}$ is a bounded sequence in ${_{0}}\mathcal{{W}}^{\alpha,p}_{\theta, \lambda}$ and there exists a subsequence $\{v_{k_j}\}_{j=1}^{\infty} \subset \{v_{k}\}_{k=1}^{\infty}$ and a function $u \in {_{0}}\mathcal{{W}}^{\alpha,p}_{\theta,\lambda}$ such that $v_{k_j} \rightharpoonup u$ in ${_{0}}\mathcal{{W}}^{\alpha,p}_{\theta,\lambda}$.
	 We need to show that $u \in {_{0}}\mathcal{{W}}^{\alpha,p}_{\theta,\lambda}$. 
	It follows from the fact that ${_{0}}\mathcal{{W}}^{\alpha,p}_{\theta,\lambda}$ is a closed  subspace of $\mathcal{W}^{\alpha,p}_{\theta,\lambda}$ and Mazur's Theorem (cf. \cite{Evans,Brezis}) that ${_{0}}\mathcal{{W}}^{\alpha,p}_{\theta,\lambda} $ is weakly closed. Hence $u \in {_{0}}\mathcal{{W}}^{\alpha,p}_{\theta,\lambda} .$
	
	Finally, since $\mathcal{E}^\alpha_{p,\theta,\lambda}$ is lower semicontinuous, then
	\begin{align*}
	\mathcal{E}^\alpha_{p,\theta,\lambda}(u) \leq \underset{k \rightarrow \infty}{\mbox{ liminf } } \mathcal{E}^\alpha_{p,\theta,\lambda}(v_k) = M.
	\end{align*}
	Thus,
	\begin{align*}
	\mathcal{E}^\alpha_{p,\theta,\lambda}(u) = \underset{v \in {_{0}}\mathcal{{W}}^{\alpha,p}_{\theta,\lambda}}{\mbox{argmin }}  \mathcal{E}^\alpha_{p,\theta,\lambda}(v) = M.
	\end{align*}
	The proof is complete. 
\end{proof}



We conclude this section with the following remark.

\begin{remark}
	The assumptions and techniques used to prove Proposition \ref{LowerSemiContinuity} and Theorem \ref{Existence} are not sharp and can be relaxed in certain cases.
	
	(i) If $\theta \in (0,1)$, then ${_{0}}\mathcal{{W}}^{\alpha,p}_{\theta,\lambda} = W^{\alpha,p}_{0}$. In this case, we can assume only that $L_{p,\theta,\lambda} (\lDa v,\rDa v,v,x) \geq c_0 \bigl(|\lDa v|^{p} + |\rDa v|^{p} \bigr) - c_1|\Ome|$ and use the fact that $c^{1-\alpha}_{\pm} = 0$ in  $W^{\alpha,p}(\Omega)$  (cf. Proposition \ref{ConstantZero}) to apply directly the fractional Poincar\'e inequality (\ref{SymPoincare}) in the proof of Theorem \ref{Existence}.
	
	(ii) If $\theta =0$ or $1$ and $\lambda=0$, then  ${_{0}}\mathcal{{W}}^{\alpha,p}_{\theta,\lambda} = {^{\pm}}{\mathring{W}}{^{\alpha,p}_{0}}(\Omega)$, and again, we can relax the condition
	 on the density function so that  $L_{p,\theta,\lambda}(\lDa v,\rDa v,v,x)  \geq c_0 \bigl(|\lDa v|^{p} + |\rDa v|^{p} \bigr) - c_1|\Omega|$ and use the fact that $ u \in {^{\pm}}{\mathring{W}}{^{\alpha,p}}(\Omega)$ to apply the fractional Poincar\'e inequality (\ref{SimplePoincare}) to prove the minimizing sequence is bounded in ${^{\pm}}{W}{^{\alpha,p}}(\Omega)$ in the proof of Theorem \ref{Existence}.
\end{remark}

\section{Fractional Calculus of Variations via Galerkin Formulation}\label{sec-5}
 
In this section, we consider the fractional calculus of variations problem:
\begin{align}\label{MinFamily}
u = \underset{ v \in  {_{0}}{\mathcal{W}}^{\alpha,p}_{\theta,\lambda}}{\mbox{argmin }} \mathcal{E}^\alpha_{p,\theta,\lambda} (v)
\end{align}
with the following generalized $p$-energy density function:  
\begin{align}\label{thetaLagrange}
    L_{p,\theta,\lambda}(\lDa v, \rDa v , v,x)  = \frac{1}{p} \Bigl( (1-\theta) |\lDa v|^{p} + \theta |\rDa v|^{p} + \lambda |v|^{p} \Bigr)  -fv 
\end{align}
for a suitably given function or functional $f$.  
We shall first derive the Galerkin formulation and the Euler-Lagrange equation  
for the associated calculus of variations problem \eqref{FundamentalMin}. We then present a 
detailed well-posedness and regularity analysis in the special case $p=2$ for the problem with both 
Dirichlet and Neumann boundary condition and various combinations of $\theta$ and $\lambda$
via the Galerkin approach. 


    

We note that it is easy to check the density function \eqref{thetaLagrange} satisfies the assumptions of Proposition \ref{LowerSemiContinuity} and Theorem \ref{Existence} for suitable $f$ (including $L^2$ functions)
and therefore, the existence of a minimizer is settled for the general case $1 < p < \infty$. 
However, our focus is to study this particular calculus of variations problem via an equivalent Galerkin
(or weak) formulation in the case $p=2$ for a weaker source function $f$. Such a Galerkin theory
will serve as a foundation for developing and analyzing efficient numerical methods for these problems \cite{Feng_Sutton3}.

\subsection{Euler-Lagrange Equation and Galerkin Formulation}\label{sec-5.1}

Before we study any well-posedness results for the problems (\ref{MinFamily}), we first discuss the associated Euler-Lagrange equation and the weak formulation.

\begin{theorem}\label{EulerLagrangeThm}
    Let $f \in L^{q}(\Omega)$ and assume that $u$ is a minimizer of (\ref{MinFamily}). Then $u$ satisfies, in the distributional sense, the following Euler-Lagrange equation:
       \begin{align}\label{EulerLagrange}
            (1-\theta) {^{-}}{\Delta}{^{\alpha}_{p}} u + \theta {^{+}}{\Delta}{^{\alpha}_{p}} u+ \lambda |u|^{p-2} u = f \text{ in } \Omega.
        \end{align}
\end{theorem}

\begin{proof}
    Since the proof is essentially the same as that of Proposition \ref{prop3.1}, we only highlight the 
    main steps. 
    Define the function $\Phi: \R_{+} \rightarrow \R$ by $\Phi(t) := \mathcal{E}^\alpha_{p,\theta,\lambda} (u + tv)$ for any $v \in C^{\infty}_{0}(\Omega)$. 
    Since $u$ is a minimizer of (\ref{MinFamily}), then $\Phi$ takes its minimum value at $t=0$, Thus, 
    $\Phi'(0)=0$, which implies that   
   \begin{align}\label{eq5.4}
   \int_{\Omega} \Bigl((1-\theta) |\lDa u|^{p-2} \lDa u \lDa v &+\theta |\rDa u|^{p-2} \rDa u \rDa v \Bigr)\,dx\\
      & +\int_{\Omega} \lambda |u|^{p-2} u v \,dx = \int_{\Omega} fv \,dx. \nonumber
    \end{align}
   Integrating by parts and using the Fundamental Lemma of the Calculus of Variations (cf. \cite{Dacorogna}) we conclude that $u$ satisfies (\ref{EulerLagrange}) in the distributional sense.  
\end{proof}

\begin{remark}
    (i) Accounting for the boundary conditions built into the energy space ${_{0}}{\mathcal{W}}^{\alpha,p}_{\theta,\lambda}$, the underlying fractional boundary value problem to  problem (\ref{MinFamily}) is
\begin{subequations}\label{BVP}
\begin{align}\label{BVPa}
        (1-\theta) {^{-}}{\Delta}{^{\alpha}_{p}} u + \theta {^{+}}{\Delta}{^{\alpha}_{p}} u+ \lambda |u|^{p-2} u = f \text{ in } \Omega, \\ 
        (1-\theta) \lT u = 0, \quad \theta \rT u = 0.
\end{align}
\end{subequations}
Here we use the same notational conventions that are detailed in Remark \ref{BoundaryNotation}.

(ii) \eqref{eq5.4} is called a weak formulation of the boundary value problem 
\eqref{BVP}.
\end{remark}

With the connection between the variational problem and the fractional boundary value problem established, 
we turn our attention to the special case $p=2$. In this case, we shall establish existence and uniqueness of minimizers via the weak formulation. To the end, we define
\begin{subequations}\label{eq5.6}
    \begin{align}\label{eq5.6a}
        a_{\theta , \lambda} (u,v) &:=   \int_{\Omega}  \Bigl((1-\theta) \lDa u \lDa v + \theta  \rDa u \rDa v +\lambda u v  \Bigr)\,dx,\\
        F(v) &: = \int_{\Omega} fv\,dx \label{eq5.6b}
    \end{align}
 \end{subequations}
It is easy to see that 
$a_{\theta , \lambda}(\cdot, \cdot): {_{0}}{\mathcal{W}}^{\alpha,2}_{\theta,\lambda}\times {_{0}}{\mathcal{W}}^{\alpha,2}_{\theta,\lambda} \rightarrow \R$  is a bilinear form 
and $F(\cdot): {_{0}}{\mathcal{W}}^{\alpha,2}_{\theta,\lambda} \rightarrow \R$ is a bounded linear functional 
for $f \in L^{2}(\Omega)$.

We are now ready to state and prove the following equivalent  theorem.

\begin{proposition}\label{WeakEqual}
    Let $u \in {_0}{\mathcal{W}}^{\alpha,2}_{\theta,\lambda} $. Then $u$ is a minimizer of 
    \eqref{MinFamily} with $p=2$ if and only if 
  \begin{align}\label{eq5.7}
  	a_{\theta, \lambda} (u,v) = F(v) \qquad \forall \, v \in {_0}{\mathcal{W}}^{\alpha,2}_{\theta,\lambda} .
  \end{align}
\end{proposition}

\begin{proof}
    Assume that $u \in {_0}{\mathcal{W}}^{\alpha,2}_{\theta,\lambda}$ is a minimizer of 
    \eqref{MinFamily}, we define 
    \[
    \Phi(t) : = \mathcal{E}^\alpha_{2,\theta,\lambda}(u+tv) \qquad \forall \, v \in {_0}{\mathcal{W}}^{\alpha,2}_{\theta,\lambda}. 
    \]
    Then $\Phi$ takes it minimum value at $t=0$. Hence, $\Phi'(0)=0$, which 
    yields \eqref{eq5.7}. 
    
    Conversely, suppose that $u \in {_0}{\mathcal{W}}^{\alpha,2}_{\theta,\lambda}$ solves 
    \eqref{eq5.7}, it is easy to check that 
    \begin{align*}
        \mathcal{E}^\alpha_{2,\theta,\lambda} (u+v)  = \mathcal{E}^\alpha_{2,\theta,\lambda} (u) 
        + \frac{1}{2} a_{\theta , \lambda} (v,v) \geq \mathcal{E}^\alpha_{2,\theta,\lambda} (u)
    \end{align*}
    for any $v \in {_0}{\mathcal{W}}^{\alpha,2}_{\theta,\lambda}$.  Thus, $u$ solves 
    \eqref{MinFamily} with $p=2$. The proof is complete. 
\end{proof}

\begin{remark}
    \eqref{eq5.7} is called a weak formulation of the boundary value problem \eqref{eq5.6}, which can 
    be formally derived from \eqref{eq5.6} by an integration by parts procedure after testing the differential equation with a function $v \in {_0}{\mathcal{W}}^{\alpha,2}_{\theta,\lambda}$. This gives a 
    precise meaning to the boundary value problem and to its solution. 
\end{remark}

\subsection{Existence and Uniqueness}\label{sec-5.2}
The goal of this subsection is to show that there exists a unique solution to the variational 
problem \eqref{eq5.7} via the well-known Lax-Milgram Theorem for each case of $\theta\in (0,1)$ and $\lambda=0$ or $1$. Together with the results of Section \ref{sec-5.1}, it proves that in the case $p=2$, there exists a unique $u \in {_0}{\mathcal{W}}^{\alpha,2}_{\theta,\lambda}$ which solves problem \eqref{MinFamily}.

\begin{proposition}\label{LaxMilgram}
    There exists a unique solution  $u \in {_0}{\mathcal{W}}^{\alpha,2}_{\theta,\lambda} $ to problem 
    \eqref{eq5.7}.
\end{proposition}

\begin{proof}
    The idea of the proof is to utilize  the Lax-Milgram Theorem. To the end,  we need to verify 
    three conditions required by the theorem. 
    \begin{itemize}
        \item[(i)] $a_{\theta, \lambda}$ is bounded in ${_0}{\mathcal{W}}^{\alpha,2}_{\theta,\lambda}   \times  {_0}{\mathcal{W}}^{\alpha,2}_{\theta,\lambda}$: there exists $M >0$ such that 
        \begin{align}\label{BiBounded}
            |a_{\theta , \lambda} (w,v)| \leq M \|w\|_{\mathcal{W}^{\alpha,2}_{\theta ,\lambda}} \|v\|_{\mathcal{W}^{\alpha,2}_{\theta ,\lambda}} \qquad\forall \, w,v\in {_0}{\mathcal{W}}^{\alpha,2}_{\theta,\lambda} .
        \end{align}
        \item[(ii)] $a_{\theta, \lambda}$ is coercive in ${_0}{\mathcal{W}}^{\alpha,2}_{\theta,\lambda}$: there exists $\gamma > 0$ such that 
        \begin{align}\label{BiCoercive}
            a_{\theta , \lambda} (v,v) \geq \gamma \|v\|_{\mathcal{W}^{\alpha,2}_{\theta ,\lambda}}^{2} \qquad\forall \, v\in {_0}{\mathcal{W}}^{\alpha,2}_{\theta,\lambda} .
        \end{align}
         \item[(iii)] $F$ is a bounded linear functional on ${_0}{\mathcal{W}}^{\alpha,2}_{\theta,\lambda}$: there exists $C > 0$ such that 
        \begin{align}\label{LinearBounded}
        |F(v)| \leq C\|v\|_{\mathcal{W}^{\alpha,2}_{\theta ,\lambda}}  \qquad\forall\, v\in {_0}{\mathcal{W}}^{\alpha,2}_{\theta,\lambda} .
        \end{align}
    \end{itemize}

    As the proof of each of these estimates depends on the solution space 
    ${_0}{\mathcal{W}}^{\alpha,2}_{\theta,\lambda}$ and its associated norm, we 
    separate the verification into subcases when necessary.

    To prove that $a_{\theta,\lambda}(\cdot , \cdot)$ is bounded and coercive, consider the 
    following cases.  
    
    {\em Case One:}  Let $\theta \in (0,1)$ and $\lambda= 0$ or $1$. In this case, ${_0}{\mathcal{W}}^{\alpha,2}_{\theta,\lambda} = H^{\alpha}_{0}(\Omega)$ which is endowed with 
    the norm  
    $$\|v \|_{H^{\alpha}_0(\Omega)} : = \Bigl(\|\lDa v\|_{L^{2}(\Omega)}^{2} + \|\rDa v \|^{2}_{L^{2}(\Omega)} \Bigr)^{\frac12}.$$ 
    The above norm is equivalent to the full $H^{\alpha}_{0}$-norm due to the fractional Poincar\'e inequality (cf. Theorem \ref{FractionalPoincareThm}). 
    
    By Schwarz inequality we get
    \begin{align*}
        |a_{\theta , \lambda}(w,v)| &= \left| \int_{\Omega} (1-\theta) \lDa w  \lDa v +\theta  \rDa  w  \rDa v + \lambda w v \,dx \right| \\
        &\leq    \|\lDa w \|_{L^{2}(\Omega)} \| \lDa v\|_{L^{2}(\Omega)} + \|\rDa w \|_{L^{2}(\Omega)} \| \rDa v\|_{L^{2}(\Omega)} \\ 
        &\qquad + \|w\|_{L^{2}(\Omega)} \|v\|_{L^{2}(\Omega)} \\
        &\leq \|w\|_{H^{\alpha}_0(\Omega)} \|v\|_{H^{\alpha}_0(\Omega)}.
     \end{align*} 
    Hence, \eqref{BiBounded} holds with $M=1$. Trivially, 
    \begin{align*}
        a_{\theta , \lambda} (v,v) &= \int_{\Omega} (1-\theta) (\lDa v)^{2} + \theta (\rDa v)^{2} + \lambda v^{2} \,dx\\
        &\geq \min \{ 1-\theta , \theta\}
        \Bigl( \|\lDa v\|_{L^{2}(\Omega)}^{2} +\|\rDa v\|_{L^{2}(\Omega)}^{2} \Bigr) \\ 
        &\geq \min \{ 1-\theta , \theta\}  \| v\|_{H^{\alpha}_0(\Omega)}^{2}.
    \end{align*}
    Thus, \eqref{BiCoercive} holds with $\gamma= \min \{ 1-\theta , \theta\}$.
    
    Lastly, the inequality \eqref{LinearBounded} follows from an application of Schwarz and fractional Poincar\'e inequality (cf. Theorem \ref{FractionalPoincareThm}) in $H^{\alpha}(\Omega)$ as follows:
     \begin{align*}
    |F(v)| =   \left|\int_{\Omega} fv\,dx\right| \leq \|f\|_{L^{2}(\Omega)} \| v\|_{L^{2}(\Omega)} \leq C_P \|f\|_{L^{2}(\Omega)}  \| v\|_{H^{\alpha}_0(\Omega)}^{2},
    \end{align*}
    where $C_p$ denotes the Poincar\'e constant.
    
    {\em Case Two:}  Let $\theta=0$ or $1$ and $\lambda =1$. In this case, we have  
    ${_0}{\mathcal{W}}^{\alpha,2}_{\theta,\lambda} = {^\pm}{H}^\alpha_0(\Omega)$ which is endowed with the norm 
    $$\|v\|_{{^\pm}{H}^\alpha_0(\Omega)}^{2} = \|\lrDa v\|_{L^{2}(\Omega)}^{2} + \|v\|_{L^{2}(\Omega)}^{2}.$$
    It is easy to see that both \eqref{BiBounded} and \eqref{BiCoercive} hold with
    $M = 1$ and $\gamma = 1$, and \eqref{LinearBounded} follows immediately from an application of 
    Schwarz inequality. 
    
   {\em Case Three:}  $\theta=0$ or $1$ and $\lambda =0$. We have 
   ${_0}{\mathcal{W}}^{\alpha,2}_{\theta,\lambda} = {^\pm}{\mathring{H}}^\alpha_0(\Omega)$ and 
    \begin{align*}
        |a_{\theta , \lambda} (w,v) | 
        \leq \| \lrDa w\|_{L^{2}(\Omega)} \|\lrDa v\|_{L^{2}(\Omega)}  
        \leq  \|w\|_{{^{\pm}}{H}{^{\alpha}_0}(\Omega)} \| v\|_{{^{\pm}}{H}{^{\alpha}_0} (\Omega)}. 
    \end{align*}
    Thus, \eqref{BiBounded}  hold with $M = 1$.  To verify \eqref{BiCoercive}, we need to
    resort to the fractional Poincar\'e inequality (cf. (\ref{SimplePoincare})  to get 
    \begin{align*}
    a_{\theta , \lambda} (v,v) 
    &= \frac12 \| \lrDa v \|_{L^{2}(\Omega)}^{2} + \frac12 \| \lrDa v \|_{L^{2}(\Omega)}^{2} \\
    &\geq \frac12 \|\lrDa v\|_{L^{2}(\Omega)}^{2} + \frac{1}{2C_p^{2}} \| v\|_{L^{2}(\Omega)}^{2}
    \geq \gamma  \|v\|_{{^{\pm}}{H}{^{\alpha}_{0}}(\Omega)}^{2},
    \end{align*}
    where $\gamma=\frac12\min\{1,C_P^{-2}\}$.  
    
    Lastly, \eqref{LinearBounded} holds for the same reason as in {\em Case One}. The proof is 
    complete.
\end{proof}

As an immediate corollary of Propositions  \ref{WeakEqual} and \ref{LaxMilgram}, we have the 

\begin{theorem}
        There exists a unique solution to problem \eqref{MinFamily} with $p=2$.  
\end{theorem}

 \begin{remark}
The well-posedness results of this subsection can be extended to inhomogeneous boundary conditions as well. In that case problem \eqref{MinFamily} becomes 
    \begin{align*}
        u = \underset{ v \in  {_{g}}{\mathcal{W}}^{\alpha,2}_{\theta,\lambda}}{\mbox{\rm argmin }} \mathcal{E}^\alpha_{2,\theta,\lambda} (v)
    \end{align*}
    where ${_{g}}{\mathcal{W}}^{\alpha,p}_{\theta,\lambda}:= \{ u \in \mathcal{W}^{\alpha,p}_{\theta,\lambda} \, :\, (1-\theta) {^{-}}{T} u = (1-\theta)g_1, \, \theta {^{+}}{T} u = \theta g_2\}$ for two given real numbers $g_1$ and $g_2$. In the case $\theta =0$ or $1$, the idea is to set $$u(x) = u_0(x) + u_g(x) := u_0 +  [(1-\theta)g_1(b-x) + \theta g_2(x-a)](b-a)^{-1}.$$
    It can be shown that $u_g \in  {_{g}}{\mathcal{W}}{^{\alpha,2}_{\theta,\lambda}}$ (for $\theta=0$ or $1$). 
    Then the problem is reduced to finding $u_0$ which is the solution to a homogeneous problem. 
\end{remark}

\subsection{Neumann Boundary Value Problems}\label{sec-5.3} In this subsection, we consider 
 if essential (Dirichlet) boundary conditions are not enforced in the energy space for problem \eqref{MinFamily}  (i.e. ${^{\pm}}{T} u$ is left free) or $*$ takes empty value in 
 problem \eqref{eq1.7}. As already demonstrated in Theorem \ref{DeriveNeumann}, 
 this implies that the homogeneous Neumann (or natural) boundary condition ${^\pm}{\mathcal{N}}^\alpha_p u=0$ 
 is imposed to the calculus of variations problem. Below we consider such prototypical 
 fractional $p$-Poisson problems, especially, for $p=2$.  As in the integer order case, 
 much of the analysis of this problem follows in a similar manner to that of its Dirichlet counterpart as presented in Sections \ref{sec-5.1} and \ref{sec-5.2}. Therefore, 
 we shall only highlight some of the consequences and differences that emerge due to 
 considering Neumann (natural) boundary conditions. 

Formally, Neumann (natural) boundary value problems allow for more freedom in their solutions than in Dirichlet (essential) boundary value problems since the traces of the solution functions do not have to be defined in order for the Neumann boundary value(s) to be defined. 
For example, if $u \in \mathcal{W}^{\alpha,p}_{\theta,\lambda}$, then $\lDa u \in L^{2}(\Omega)$ and $\rlIoa \lrDa u \in C(\overline{\Omega})$. Therefore, the mapping,  $\mathcal{W}^{\alpha,p}_{\theta,\lambda} \mapsto {^{\pm}}{\mathcal{N}}{^{\alpha}_p}( \mathcal{W}_{\theta,\lambda}^{\alpha,p})$, is well defined for any $(\alpha ,p) \in (0,1) \times [1, \infty]$. Consequently, unlike the Dirichlet case, the restriction $\alpha p > 1$ is not needed. 
Moreover, in the integer order case, a Neumann boundary value problem often requires a side-condition (or compatibility-condition) to ensure the uniqueness of solutions. However, in the 
fractional order case, such a side-condition is not needed. 

\begin{theorem}
    Assume that 
    \begin{align}\label{NeumannProb}
    u = \underset{v \in \mathcal{W}^{\alpha,p}_{\theta,\lambda}}{\mbox{\rm argmin }} \mathcal{E}^\alpha_{p,\theta,\lambda} (v).
    \end{align}
    Then $u$ satisfies  equation \eqref{EulerLagrange} and 
    the Neumann boundary conditions 
    \begin{align}\label{NaturalBoundaryConditions}
            (1-\theta) {^{-}}{\mathcal{N}}{^{\alpha}_{p}} u = 0,\qquad 
            \theta {^{+}}{\mathcal{N}}{^{\alpha}_{p}} u =0
    \end{align}
    in the distributional sense. 
\end{theorem}

\begin{proof}
	The validity of \eqref{EulerLagrange} can be proved in exactly the same way as in
	the proof of Theorem \ref{EulerLagrangeThm} because that proof does not require the zero-boundary condition of the assumed minimizer (in that problem). Similarly, 
	using the same techniques as in the proof of Theorem \ref{DeriveNeumann} we can show
	that $u$ satisfies the Neumann boundary conditions \eqref{NaturalBoundaryConditions}
	in the distributional sense.
\end{proof}

\begin{remark}
    (i) The associated fractional PDE problem to \eqref{NeumannProb} is the following fractional Neumann boundary value problem:
    \begin{subequations}\label{NBVP}
    \begin{align}
            (1-\theta) {^{-}}{\Delta}{^{\alpha}_{p}} u + \theta {^{+}}{\Delta}{^{\alpha}_{p}} u + \lambda|u|^{p-2}u = f  \text{ in } \Omega,& \\ 
            (1-\theta) {^{-}}{\mathcal{N}}{^{\alpha}_{p}} u = 0, \quad \theta {^{+}}{\mathcal{N}}{^{\alpha}_{p}}  u = 0.&
    \end{align}
    \end{subequations}
    
    (ii) Unlike the Dirichlet problem, the above Neumann boundary value problem may be well defined for any $(\alpha , p) \in (0,1) \times [1, \infty]$ since we do not require the function trace to exist; hence we need not require $\alpha p >1$. 
\end{remark}

It can be shown using the same techniques as in Sections \ref{sec-5.2} and \ref{sec-5.3} that 
the Neumann boundary value problem is well-posed when $p=2$.  We skip the proof and 
leave it to the interested reader. Moreover, we note that the well-posedness of the Neumann problem \eqref{NeumannProb} with $p=2$ does not require a side-condition for uniqueness.

\begin{proposition}
    $u \in \mathcal{W}^{\alpha,2}_{\theta , \lambda}$ is a solution of problem \eqref{NeumannProb}
    if and only if $u$ satisfies 
    \begin{align}\label{eq5.14}
     a_{\theta, \lambda} (u,v) = F(v) \qquad \forall \, v \in \mathcal{W}^{\alpha,2}_{\theta , \lambda},
    \end{align}  
    where $a_{\theta,\lambda}(\cdot,\cdot)$ and $F(\cdot)$ are defined by \eqref{eq5.6}.
\end{proposition}


\begin{theorem}
There exists a unique solution $u \in \mathcal{W}^{\alpha,2}_{\theta, \lambda}$ to 
problem \eqref{eq5.14}. Hence, problem \eqref{NeumannProb} is well-posed with $p=2$. 
\end{theorem}



\subsection{Calculus of Variations in the Riesz Fractional Derivative}\label{sec-5.4}
In this subsection, we consider fractional calculus of variations problems 
which involve the Riesz fractional derivative $\zDa v$. 
For the sake of clarity, we only consider the Dirichlet $2$-energy case (i.e., $p=2$), where
\begin{align}
   \mathcal{E}^{\alpha}_{z,\lambda} (v) := \frac{1}{2} \int_{\Omega} |\zDa v|^{2} 
    + \lambda |v|^{2} \,dx  - \langle f , v \rangle,
\end{align}
for given $f \in ({^{z}}{H}{^{\alpha}_{0}}(\Omega))^{*}$ and $\lambda=0$ or $1$.

Our goal here is to prove the well-posedness of the minimization problem
\begin{align}\label{RieszMin}
    u =\underset{ v \in {^{z}}{H}{^{\alpha}_{0}}(\Omega) }{\mbox{argmin }} \mathcal{E}^{\alpha}_{z,\lambda} (v).
\end{align}
We note that the energy space is now the Riesz space ${^{z}}{H}{^{\alpha}_{0}}(\Omega)$,
which is significantly different from the one-sided spaces ${^{\pm}}{H}{^{\alpha}_{0}}(\Omega)$.
We shall again proceed by deriving an equivalent weak formulation and finishing 
the proof by using Lax-Milgram theorem. To the end, we first define  
the bilinear form 
$a_{z,\lambda}: {^{z}}{H}{^{\alpha}_{0}} (\Omega) \times {^{z}}{H}{^{\alpha}_{0}}(\Omega) \rightarrow \R$ by 
    \begin{align*}
        a_{z,\lambda} (w,v) := \int_{\Omega}  \zDa w \zDa v + \lambda uv  \,dx,
    \end{align*}
    and the linear functional $F(v)=  \langle f, v\rangle$. 
 

\begin{theorem}\label{UniqueRiesz}
 Problem \eqref{RieszMin} has a unique solution $u \in {^{z}}{H}{^{\alpha}_{0}}(\Omega)$.  
\end{theorem}

\begin{proof} We consider the cases $\lambda=0$ and $1$ separately.
If $\lambda =1$ and $f \in ({^{z}}{H}{^{\alpha}_{0}}(\Omega))^{*}$, it can be shown that a function $u \in {^{z}}{H}{^{\alpha}_{0}}(\Omega)$ solves \eqref{RieszMin} if and only if it satisfies 
\begin{equation}\label{eq5.17}
a_{z,\lambda}(u,v) = \langle f, v\rangle \qquad \, \forall v \in {^{z}}{H}{^{\alpha}_{0}}(\Omega).
\end{equation}
Moreover, it is easy to show that $a_{z,\lambda}(\cdot, \cdot)$ is bounded and coercive on  ${^{z}}{H}{^{\alpha}_{0}}(\Omega) \times {^{z}}{H}{^{\alpha}_{0}}(\Omega)$ and 
$\langle f, \cdot\rangle $ is a bounded linear functional on ${^{z}}{H}{^{\alpha}_{0}}(\Omega)$ which
is endowed with the norm $\|u\|_{{^{z}}{H}{^{\alpha}}(\Omega)}^{2} := \|u\|_{L^{2}(\Omega)}^{2} + \| \zDa u\|_{L^{2}(\Omega)}^{2}$. 
By Lax-Milgram theorem, we obtain the desired well-posedness.
    
If $\lambda = 0$, 
the absence of the zero order term and the lack of a fractional Poincar\'e inequality in the space ${^{z}}{H}{^{\alpha}_{0}}(\Omega)$ causes a difficulty to establish the 
coercivity of the bilinear form $a_{z,\lambda} (\cdot,\cdot)$ on ${^{z}}{H}{^{\alpha}_{0}}(\Omega)$ with the norm given above. 
To sidestep the difficulty, we appeal to Proposition \ref{RieszNorm}, which shows that
$\| \zDa v \|_{L^{2}(\Omega)}=a_{z,\lambda} (v,v)^{\frac12}$ is in fact 
a norm in ${^{z}}{H}{^{\alpha}_{0}}(\Omega)$ for $\alpha<1$.
So we endow the space ${^{z}}{H}{^{\alpha}_{0}}(\Omega)$ with this bilinear-form 
induced norm and assume that $f \in ({^{z}}{H}{^{\alpha}_{0}}(\Omega))^{*}$, the dual 
space of ${^{z}}{H}{^{\alpha}_{0}}(\Omega)$ with the induced norm. 
The boundedness of the bilinear form follows immediately from using Schwarz inequality (with 
$M=1$). Thus, the well-posedness follows again from an application of Lax-Milgram theorem. 
\end{proof}

\begin{remark}
	(i) Although the above theorem ensures the well-posedness of problem \eqref{RieszMin}
	in ${^{z}}{H}{^{\alpha}_{0}}(\Omega)$, in both cases $\lambda=0$ and $1$, the solution 
	estimates are slightly different as it is measured in different norms.

    (ii) Notice that $({^{z}}{H}{^{\alpha}_{0}}(\Omega))^*\subset H^{-\alpha}(\Omega)$ because  $H^{\alpha}_{0}(\Omega) \subset {^{z}}{H}{^{\alpha}_{0}}(\Omega)$. It follows from Theorem \ref{UniqueRiesz} that there exists a unique $u_f \in {^{z}}{H}{^{\alpha}_{0}}(\Omega)$ that solves
    \eqref{eq5.17} for a given $f \in ({^{z}}{H}{^{\alpha}_{0}}(\Omega))^{*}$.  On the other hand, 
    restricting the test function $v\in H^{\alpha}_{0}(\Omega)$ in \eqref{eq5.17} and
    repeating the proof we can show that \eqref{eq5.17} has a unique solution 
    $\hat{u}_f \in H^{\alpha}_{0}(\Omega)$. We now show that $u_f=\hat{u}_f$.
    First, noticing that $u_f - \hat{u}_f = c\kappa^{\alpha}_{z}$ for some $c\in\mathbb{R}$. 
    It suffices to show that $c=0$. 
    Second, since $u_f, \hat{u}_f \in L^{2}(\Omega)$, so does $u_f - \hat{u}_f\in L^{2}(\Omega)$. Finally, if $c \neq 0$, $\|u_f - \hat{u}_f\|_{L^{2}(\Omega)} = |c| \| \kappa^{\alpha}_{z}\|_{L^{2}(\Omega)} = \infty$, which contradicts the fact that $u_f - \hat{u}_f\in L^{2}(\Omega)$. Therefore, $c=0$  and $ u_f= \hat{u}_f$ almost everywhere in $\Omega$. Thus, $u_f$ in fact belongs to $H^{\alpha}_{0}(\Omega)$.
\end{remark}
    
\subsection{Some Regularity Results of One-sided Poisson Problems}\label{sec-5.5}
In this subsection, we examine regularities of one-sided Poisson problems. 
In Section \ref{sec-4.2.2} we proved a related fractional Calder\'on-Zygmund type result. In that case, we examined how the fractional Laplace operator is related to differentiating twice in a single direction. In this subsection, we instead show how the regularity of the data function $f$ in (\ref{RegDiffEq}) effects the regularity of our weak solution. For our purpose, we restrict our attention 
to the case $p=2$. It has been shown in the previous sections that any $u \in {_*}{\mathcal{W}}^{\alpha,2}_{\theta,\lambda}$ 
that minimizes $\mathcal{E}^{\alpha}_{2,\theta,\lambda}$ is a weak solution of

\begin{subequations} \label{RegDiffEq}
\begin{align}
        (1- \theta) \lLa u + \theta \rLa u + \lambda u = f \mbox{ in } \Omega,&\\ 
        (1-\theta) \left( \lT \text{ or }\lNa \right) u = 0, \quad \theta \left( \rT \text{ or } \rNa\right) u = 0.&
\end{align}
\end{subequations}

Its weak formulation is given by \eqref{eq5.7} or \eqref{eq5.14}. That is, find $u \in {_*}{\mathcal{W}}^{\alpha,2}_{\theta,\lambda}$ 
such that 
\begin{align}\label{eq5.19}
    a_{\theta , \lambda} (u,v) = F(v) \qquad \forall  \, v \in {_*}{\mathcal{W}}^{\alpha,2}_{\theta , \lambda}.
\end{align}

On one hand, noting that if $f \in L^{2}(\Omega)$, then there holds that $\lLa u \in L^{2}(\Omega)$ and $ \rLa u \in L^{2}(\Omega)$.
On the other hand, unlike the integer order case, we do not expect that $u \in \lrHta (\Omega)$ in general (cf. \cite{Feng_Sutton2}) because each one-sided fractional Laplacian involves one-sided derivatives in both directions, instead of two derivatives in a single direction (cf. Proposition \ref{Directions}). In this case, $u$ and it's left/right derivative live in different spaces relative to the direction of differentiation as the next theorem shows. Due to the nature of alternating directions in the fractional Laplacian(s) presented, we introduce a new function space, 
\begin{align}\label{LaplaceSpace}
{^{\pm}}{\mathcal{S}}{^{\alpha}_{n}} := \{ u \in L^{2}(\Omega) \, : \, ( {^{\pm}}{\Delta}{^{\alpha}})^{n} u \in L^{2}(\Omega)\},
\end{align}
where $( {^{\pm}}{\Delta}{^{\alpha}})^{n}$ is understood as the composition of $n$ fractional Laplace operators. 

\begin{theorem}\label{Regularity}
Let $f \in {^{\pm}}{\mathcal{S}}{^{\alpha}_{n}} \cup V$ for a given Sobolev space $V \subset L^{2}(\Omega)$ and $u \in \lrHa (\Omega)$ be a weak solution of \eqref{RegDiffEq} for $\theta =0$ or $1$. If $\lambda =1$, then $u \in {^{\pm}}{\mathcal{S}}{^{\alpha}_{n+1}}$. If $\lambda =0$, then $u \in {^{\pm}}{\mathcal{S}}{^{\alpha}_{n+1}}$ and ${^{\pm}}{\Delta}{^{\alpha}} u \in V$. 
\end{theorem}

\begin{proof}
Let $g := f-\lambda u$. 
Since $u$ is a weak solution, it must satisfy 
\eqref{eq5.19}. By the fact that $C^{\infty}_{0}(\Omega) \subset \lrHa(\Omega)$, 
we get 
\begin{align*}
    \int_{\Omega} \lrDa u \lrDa \varphi\,dx = \int_{\Omega} g \varphi \,dx \qquad \forall \, \varphi \in C^{\infty}_{0}(\Omega). 
\end{align*}
It follows from the definition of weak fractional derivatives (see Appendix A) that 
$\rlDa  (\lrDa u)$ exists and equals $g$. If $\lambda =0$, $g \equiv f \in {^{\pm}}{\mathcal{S}}{^{\alpha}_{n}} \cup V$. Hence ${^{\pm}}{\Delta}{^{\alpha}} u \in {^{\pm}}{\mathcal{S}}{^{\alpha}_{n}} \cup V$, implying that $ u \in {^{\pm}}{\mathcal{S}}{^{\alpha}_{n+1}}$. If $\lambda =1$, if follows by the assumption $f \in {^{\pm}}{\mathcal{S}}{^{\alpha}_{n}}$ and a bootstrapping argument that $u \in {^{\pm}}{\mathcal{S}}{^{\alpha}_{n+1}}$. 
\end{proof}

\begin{remark}
(i) In Theorem \ref{Regularity}, one may consider the space $V$ as ${^{\pm}}{H}{^{\beta}}(\Omega)$ or ${^{\mp}}{H}{^{\beta}}(\Omega)$ for any $\beta \geq 0$. Other spaces could be considered, but these are the most natural selections.

(ii) The case $\lambda =1$ is clearly more delicate. We see that this case is, in general, unaffected by the assumption $f \in V$. This is due to the restriction that $u$ places on boosting the regularity. In this case, the order in which the differing directions of differentiation are applied plays a major role.  For example, regardless of the assumptions on $f$, we cannot conclude that ${^{\pm}}{\Delta}{^{\alpha}} u \in {^{\mp}}{H}{^{\alpha}}(\Omega)$ because in general $\rlDa u$ may not exist.

(iii) Theorem \ref{Regularity} is the fractional counterpart (or generalization) of the well-known regularity result for solutions to the integer Poisson equation. Formally, Theorem \ref{Regularity} recovers the integer result when $\alpha \rightarrow 1$. Clearly, in that case, things are simplified because there is no notion of direction built into the derivative (or Sobolev space) definition(s). 

(iv) The regularity in the case $\theta \in (0,1)$ cannot be proven in a similar way and has not been well understood at this point. 
\end{remark}

    Finally, we consider the regularity of solutions to the Riesz problem \eqref{RieszMin}, in which 
    the zero-order term plays an important role. 
 
    \begin{theorem}
    Let $u$ be the unique weak solution to problem \eqref{RieszMin} with $\lambda =1$. If $f \in L^{2}(\Omega)$, then $\zDa u \in \zHa (\Omega)$. 
    \end{theorem}

    \begin{proof}
        By the definition of $u$ we have 
        \begin{align*}
            \int_{\Omega} \zDa u \zDa v\,dx = \int_{\Omega} (f - \lambda u) v\,dx \qquad \forall \, v \in C^{\infty}_{0}(\Omega).  
        \end{align*}
        It follows from Definition \ref{WeakDerivative} that $\zDa (\zDa u) = f- \lambda u$ almost everywhere in $\Omega$. Thus, $\zDa u \in \zHa(\Omega)$. 
    \end{proof}
    
    \begin{remark}
    In the case $\lambda =0$, we only get that  $\zDa (\zDa u) = f$ in the distributional sense. Thus, $\zLa u$ exists as a distribution. However, we cannot elevate the regularity due to the need for $f \in ({^{z}}{H}{^{\alpha}_{0}}(\Omega))^*$ because our lacking a fractional Poincar\'e inequality in the 
    space $ {^{z}}{H}{^{\alpha}_{0}}(\Omega)$.
    \end{remark}


\section{Conclusion}\label{sec-6}

In this paper we systematically studied one-dimensional pure calculus of variations problems 
in the form of \eqref{FundamentalMin}. Through these families of problems, we introduced and studied new notions of one-sided fractional $p$-Laplacian(s) and associated fractional Neumann boundary operators. Unlike any existing definitions, these are understood through the weak fractional derivative (cf. \cite{Feng_Sutton1}) and are consistent with the variational structure. The existence of solutions to \eqref{FundamentalMin} were proved via direct methods and the special case when $p=2$ was proven to be well-posed via a Galerkin formulation. Each of these were proven in the natural setting of newly developed fractional Sobolev spaces (cf. \cite{Feng_Sutton2}). Additionally, some regularity results were proven for the one-sided problems.

It is expected that this work (and \cite{Feng_Sutton1, Feng_Sutton2}) will lay down a theoretical foundation for developing efficient numerical methods for fractional calculus of variations 
problems and related PDEs in the form \eqref{EulerLagrange}.  
Moreover, we hope that this work will also stimulate more research on and applications of the fractional calculus of variations problems with more general energy functionals.

\appendix

\section{Weak Fractional Derivatives} \label{App.1}

In this appendix, we recall the definitions and basic properties of weak fractional derivatives, and 
refer the reader to  \cite{Feng_Sutton1} for the characterization theorem(s) and their properties such as product and chain rules that are necessary for a rich calculus.  We use ${^{-}}{D}{^{\alpha}}$ and ${^{+}}{D}{^{\alpha}}$ to denote respectively any left and right $\alpha$-order classical derivative
including Riemann-Liouville, Caputo, Fourier, and Gr\"unwald-Letnikv derivative. We note that all these 
derivative concepts are equivalent on the space $C^{\infty}_{0}(\Omega)$. We also use $\tilde{\varphi}$ to denote the zero extension on $\mathbb{R}$ of any   $\varphi\in C^{\infty}_{0}(\Omega) $.

\begin{definition}\label{WeakDerivative}
        For $\alpha> 0$, let $[\alpha]$ denote the integer part of $\alpha$. Let $u \in L^{1}(\Omega)$, 
       \begin{itemize} 
       \item[{\rm (i)}] a function $v \in L_{loc}^{1} (\Omega)$ is called the left weak fractional derivative of $u$ if 
        \begin{align*}
            \int_{\Omega} v(x) \varphi(x) \,dx = (-1)^{[\alpha]} \int_{\Omega} u(x) {^{+}}{D}{^{\alpha}} \tilde{\varphi}(x) \, dx
             \qquad \forall \, \varphi \in C_{0}^{\infty} (\Omega),
        \end{align*}
        and we write ${^{-}}{ \mathcal{D}}{^{\alpha}} u:=v$; 
     \item[{\rm (ii)}] a function $w\in L_{loc}^{1} (\Omega)$ is called the right weak fractional derivative of $u$ if 
      \begin{align*}
       \int_{\Omega} w(x) \varphi(x) \,dx = (-1)^{[\alpha]} \int_{\Omega} u(x) {^{-}}{D}{}^{\alpha} \tilde{\varphi}(x) \,dx
       \qquad \forall \, \varphi \in C_{0}^{\infty} (\Omega), 
      \end{align*}
       and we write ${^{+}}{\mathcal{D}}{^{\alpha}} u:=w$. Additionally, the Riesz weak fractional derivative is defined as ${^{z}}{\mathcal{D}}{^{\alpha}} u : = \frac{1}{2}( {^{-}}{\mathcal{D}}{^{\alpha}}u + {^{+}}{\mathcal{D}}{^{\alpha}} u)$.
      \end{itemize}
   \end{definition}
   
   It was proved  in \cite{Feng_Sutton1} that Definition \ref{WeakDerivative} is well defined. Many basic properties of weak fractional derivatives hold, including linearity, semigroup rules, and consistency with lower and higher order derivatives. Some properties, such as semigroup rules, do not follow directly from the definition and are nontrivial. We refer the interested reader to  \cite{Feng_Sutton1} for details. 
    
   \begin{proposition}[cf. \cite{Cai}]\label{NullRiesz}
    Let $0 < \alpha <1$ and $\Omega = (a,b) \subset \R$. Then the null space of the Riesz fractional derivative operator, ${^{z}}{\mathcal{D}}{^{\alpha}},$ is given by 
    \begin{align*}
        \mathcal{N} ( {^{z}}{\mathcal{D}}{^{\alpha}})  &= \mbox{\rm span}\{ \kappa^{\alpha}_{z_1}, \kappa^{\alpha}_{z_2}\} \\
        &: =\mbox{\rm span} \left\{ (x-a)^{\alpha/2} (b-x)^{\alpha/2 -1}, (x-a)^{\alpha /2 -1} (b-x)^{\alpha/2} \right\}. 
    \end{align*}
  \end{proposition}
    
    Next, we cite the important \textit{Fundamental Theorem of weak Fractional Calculus} (FTwFC) for  finite domains from the weak fractional calculus theory (cf.  \cite{Feng_Sutton1}). 

    \begin{theorem}\label{FTWFC}
       Let $ 0 < \alpha < 1$, $ p\in [1,\infty]$, then for any $u \in L^{p}((a,b))$ with $\lrDa u \in L^{p}((a,b))$, there holds
       \begin{align}\label{WeakFTFC}
           u(x) = c^{1-\alpha}_{\pm} \kappa^{\alpha}_{\pm}(x)  + \lrIa \lrDa u(x)
       \end{align}
       for almost every $x \in  (a,b)$ where 
       \begin{align*}
            c^{1-\alpha}_{-} := \dfrac{\lIoa u (a)}{ \Gamma(1-\alpha)} , \quad  c^{1-\alpha}_{+} := \dfrac{\rIoa u(b)}{ \Gamma(1-\alpha)},
       \end{align*}
       and 
       \begin{align*}
           \kappa_{-}^{\alpha}(x) = (x-a)^{\alpha-1}, \quad \kappa_{+}^{\alpha}(x) = (b-x)^{\alpha -1}.  
       \end{align*}
   \end{theorem}
   
   \begin{remark}
   It is not known whether the Riesz fractional derivative satisfies a similar fundamental theorem of calculus. Lacking such a powerful fundamental theorem  is the main reason to make Riesz type problems difficult to analyze. 
   \end{remark}

\section{Fractional Sobolev Spaces}\label{App.2} 
In this appendix we cite the basic definitions and properties of weak fractional Sobolev spaces  and refer the interested reader to \cite{Feng_Sutton2} for the details and the complete theory.  
   
   \begin{definition}\label{Sobolev}
       For $\alpha>0$, let $m :=[\alpha]$. For $1 \leq p \leq \infty$, the left/right fractional Sobolev space ${^{\pm}}{W}{^{ \alpha , p}} (\Omega)$ is defined by  
        \begin{align} \label{FSS}
             \lrWap (\Omega) := \left\{ u \in W^{m,p}(\Omega): 
             \lrDa   u \in L^{p}(\Omega) \right\},
        \end{align}
        which are endowed respectively with the norms 
        \begin{align} \label{FSS_norm}
          \|u\|_{\lrWap (\Omega)}:= \begin{cases}
          \left(\left\|u\right\|_{W^{m,p}(\Omega)}^{p} + \left\|\lrDa u \right\|_{L^{p}(\Omega)}^{p} \right)^{\frac{1}{p}} &\qquad \text{if } 1 \leq p < \infty,\\
          \|u\|_{W^{m,\infty}(\Omega)} 
          + \left\|\lrDa u \right\|_{L^{\infty}(\Omega)} &\qquad \text{if } p = \infty.
          \end{cases} 
        \end{align}
       
    \end{definition}
    
    \begin{remark}
       When $0 < \alpha < 1$ (i.e., $m=0$) and $1\leq p < \infty$, we have  $$\lrWap(\Omega) := \bigl\{ u \in L^{p}(\Omega):\, \lrDa u\in L^{p}(\Omega) \bigr\}$$ with the norm, $$\|u \|_{\lrWap (\Omega)} : = \left( \|u\|_{L^{p}(\Omega)}^{p} + \|\lrDa u \|_{L^{p}(\Omega)}^{p}\right)^{\frac{1}{p}}.$$
       \end{remark}
       
       In addition to the one-sided spaces $\lrWap (\Omega)$, we also define so-called symmetric 
       fractional order Sobolev space as 
        \begin{align}\label{SFSS}
            W^{\alpha,p}(\Omega):= \lWap (\Omega) \cap \rWap (\Omega),
        \end{align}
        which is endowed with the norm
        \begin{align}\label{SFSS_norm}
            \|u\|_{W^{\alpha,p}(\Omega)} &:= \begin{cases} 
            \left( \|u\|_{\lWap(\Omega)}^{p} + \|u\|_{\rWap (\Omega)}^{p} \right)^{\frac{1}{p}}   &\qquad \text{if } 1\leq p < \infty,\\
             \|u\|_{{^{-}}{W}{^{\alpha,\infty}}(\Omega)} + \|u \|_{{^{+}}{W}{^{\alpha,\infty}}(\Omega)}  &\qquad \text{if } p = \infty.
             \end{cases}
        \end{align}
 
   Below we cite several elementary properties of the spaces ${^{\pm}}{W}{^{\alpha,p}}$ and $W^{\alpha,p}$ and refer the interested reader to \cite{Feng_Sutton2} for their proofs and the discussion of other more advanced properties.

   \begin{proposition} \
    \begin{itemize}
        \item [\rm{(i)}] For $\alpha >0$, $1 \leq p \leq \infty$, $\left\| \cdot \right\|_{{^{\pm}}{W}{^{\alpha , p}}(\Omega)}$ and $\|\cdot\|_{W^{\alpha,p}(\Omega)}$ are  norms  on ${^{\pm}}{W}{^{\alpha , p}}(\Omega)$ and $W^{\alpha,p}(\Omega)$ respectively,
        \item [\rm{(ii)}] ${^{\pm}}{W}{^{\alpha,p}}(\Omega)$ and $W^{\alpha,p}(\Omega)$ are Banach spaces with these norms,
        \item[\rm{(iii)}] Endowed respectively with the inner products
        $\langle u,v \rangle_{\pm} : =(u,v) +  \left( {^{\pm}}{\mathcal{D}}{^{\alpha}} u , {^{\pm}}{\mathcal{D}}{^{\alpha}}v \right), $
         ${^{\pm}}{W}{^{\alpha,2}}(\Omega)$ and $W^{\alpha,2}(\Omega)$ are Hilbert spaces. In this case, we adopt the standard notations ${^{\pm}}{H}{^{\alpha}}(\Omega) : = {^{\pm}}{W}{^{\alpha,2}}(\Omega)$ and $H^{\alpha}(\Omega):= W^{\alpha,2}(\Omega)$,
         \item[\rm{(iv)}] ${^{\pm}}{W}{^{\alpha,p}}(\Omega)$ and $W^{\alpha,p}(\Omega)$ are reflexive for $1<p<\infty$ and separable for $1\leq p < \infty$. 
    \end{itemize}
     
\end{proposition}

    Finally, we introduce the Riesz type fractional Sobolev spaces.
    
    \begin{definition}
       For $0 < \alpha < 1$, $1 \leq p < \infty$, the Riesz fractional Sobolev spaces $\zWap(\Omega)$ are defined by  
        \begin{align} \label{RieszSS}
             \zWap(\Omega) = \left\{ u \in L^{p}(\Omega): 
             \zDa  u \in L^{p}(\Omega) \right\},
        \end{align}
        which is endowed with the norm 
        \begin{align} \label{RieszSS_norm}
          \|u\|_{\zWap(\Omega)} := ( \|u\|_{L^{p}(\Omega)}^{p} + \|\zDa u \|_{L^{p}(\Omega)}^{p} )^{1/p}.
        \end{align}
        Moreover, $\zHa(\Omega) : = {^{z}}{W}{^{\alpha,2}}(\Omega)$ is endowed with 
        the inner product
    \begin{align}\label{RieszInnerProduct}
        (u,v)_{z} := (u,v)_{L^{2}(\Omega)} + (\zDa u  , \zDa v )_{L^{2}(\Omega)},
    \end{align}  
    It is easy to check that $\zWap(\Omega)$ is a Banach space and $\zHa (\Omega)$ is a Hilbert space.
\end{definition}

    Another concept that plays a crucial role in our study is that of function traces. Unlike integer order spaces,  the trace concept is one-sided and direction dependent in the fractional Sobolev spaces ${^{\pm}}{W}{^{\alpha,p}}(\Omega)$. This is a unique property of these spaces which have major impacts in the types of boundary conditions we can consider for one-sided fractional differential equations and the calculus of variations problems.
    
    \begin{definition}\label{trace}
        We define trace operator $\lT: \lWap((a,b))\to \R$ by $\lT u=\lT u|_{x=b} := u(b)$ and define trace operator $\rT: \rWap((a,b))\to \R$ by $\rT u=\rT u|_{x=a} := u(a)$.  
    \end{definition}
    
    \begin{remark}
        We note that the above trace concept is a consequence of a compact embedding result for one-sided spaces ${^{\pm}}{W}{^{\alpha,p}}$.
        It can be shown that when $c_{\pm}^{1-\alpha} = 0$, functions have trace values at both ends of the domain/interval. Such a characteristic forces us to consider additional fractional Sobolev spaces. 
    \end{remark}
    
   \begin{definition}
   	Define the following space
   \begin{align}\label{mathringspace}
       \lrcWap(\Omega) : = \bigl\{ u \in \lrWap (\Omega)\,:\, c^{1-\alpha}_{\pm} = 0 \bigr\}
   \end{align}
   with the traditional notation $\lrcHa(\Omega) : = {^{\pm}}{\mathring{W}}{^{\alpha,2}}(\Omega)$. 
   \end{definition}

    
    \begin{definition}\label{ZeroTraceSpaces}
        Let $0 < \alpha <1$ and $1 < p <\infty$. Suppose that  $\alpha p >1$.  Define  
        \begin{align*}
            {^{\pm}}{W}{^{\alpha,p}_{0}}(\Omega) &:= \bigl\{ u \in \lrWap(\Omega) \,:\, \lrT u = 0 \bigr\},\\
            W^{\alpha,p}_{0}(\Omega) &:= \bigl\{ u \in W^{\alpha,p}(\Omega) : \lT u=0 \mbox{ and } \rT u = 0  \bigr\}.
        \end{align*}
    \end{definition}

\begin{proposition}
    Let $1\leq p  <\infty$ and $u \in {^{\pm}}{W}{^{\alpha,p}_{0}}(\Omega)$. Then $\| u \|_{{^{\pm}}{W}{^{\alpha,p}_{0}}(\Omega)} := \|\lrDa u\|_{L^{p}(\Omega)}$ defines a norm on ${^{\pm}}{W}{^{\alpha,p}_{0}}(\Omega)$.  Similarly, 
   $\|u \|_{{^{\pm}}{\mathring{W}}{^{\alpha,p}}(\Omega)} := \|{^{\pm}}{\mathcal{D}}{^{\alpha}} u \|_{L^p(\Omega)}$ defines a norm. 
\end{proposition}

\begin{proposition}\label{ConstantZero}
If $u \in {W}^{\alpha,p}$, then $\rT \lIa u = \lT \rIa u = 0$. That is, $c^{1-\alpha}_{+} = c^{1-\alpha}_{-} = 0.$
\end{proposition}

\begin{remark}
Proposition \ref{ConstantZero} ensures us that $W^{\alpha , p }_{0}(\Omega) = \mathring{W}^{\alpha,p}_{0}(\Omega)$. Therefore, we do not differentiate these two spaces in the way we do for one-sided spaces. 
\end{remark}

    A final set of results below will play a pivotal roll in proving well-posedness in Section \ref{sec-5}. The first one is a fractional Poincar\'e inequality.
    
     \begin{theorem}\label{FractionalPoincareThm} 	
   	Fractional Poincar\'e Inequality: Let $0 < \alpha <1$ and $1 \leq p < \infty$. Then there exists a constant $C = C(\alpha, \Omega)>0$ such that
   	\begin{align}\label{FractionalPoincare}
   	\|u - c_{\pm}^{1-\alpha}\kappa_{\pm}^{\alpha}\|_{L^{p}(\Omega)} \leq C\| \lrDa u\|_{L^{p}(\Omega)}
   	\qquad \forall \, u \in \lrWap(\Omega)
   	\end{align}
   	and 
   	\begin{align}\label{SimplePoincare}
       \|u\|_{L^{p}(\Omega)} \leq C\|\lrDa u \|_{L^{p}(\Omega)} \qquad \forall \, u \in \lrcWap(\Omega).
   \end{align}
   Moreover,
   	\begin{align}\label{SymPoincare}
   	    \|u\|_{L^{p}(\Omega)} \leq C\left( \|\lDa u \|_{L^{p}(\Omega)} + \|\rDa u\|_{L^{p}(\Omega)} \right) \quad \forall \, u \in W^{\alpha,p}(\Omega).
   	\end{align}
   	
   \end{theorem}

    \begin{proposition}\label{RieszNorm}
        $\|{^{z}}{\mathcal{D}}{^{\alpha}} u\|_{L^{p}(\Omega)}$ defines a norm on ${^{z}}{W}{^{\alpha,p}}(\Omega)$ if $(2 - \alpha)p > 2$. 
    \end{proposition}
    \begin{proof}
        We need only check that if $\|{^{z}}{\mathcal{D}}{^{\alpha}} u \|_{L^p(\Omega)} = 0$, then $u = 0$. That is, we must show that if $(2-\alpha )p > 2$, then $\mathcal{N}({^{z}}{\mathcal{D}}{^{\alpha}}) = \{0\}$. By Proposition \ref{NullRiesz}, we know that in general, $\mathcal{N}({^{z}}{\mathcal{D}}{^{\alpha}})= \left\{ (x-a)^{\alpha/2} (b-x)^{\alpha/2 -1}, (x-a)^{\alpha /2 -1} (b-x)^{\alpha/2} \right\}.$ Then see that for any $c \in (a,b)$,
        \begin{align*}
            \int_{a}^{b} (x-a)^{\alpha p/2} (b-x)^{(\alpha /2 -1)p} \,dx \geq (c-a)^{\alpha p/2} \int_{c}^{b} (b-x)^{(\alpha /2 -1)p} \,dx 
        \end{align*}
        where the lower bound is unbounded under the assumption $(2-\alpha) p > 2$. The calculation for $(x-a)^{\alpha /2 -1} (b-x)^{\alpha/2}$ is similar. Hence, if $(2-\alpha )p > 2$, then $\mathcal{N}( {^{z}}{\mathcal{D}}{^{\alpha}}) = \{0\}$. This completes the proof. 
    \end{proof}
    
    \begin{remark}
        In the particular case $p=2$, we have that $\|\zDa u\|_{L^{2}(\Omega)}$ defines a norm on the space $\zHa (\Omega)$. 
    \end{remark}
   
   Finally, we have a precompactness result essential for our study of the direct method in the Fractional Calculus of Variations in Section \ref{sec-4}.
   \begin{lemma}\label{Precompact}
If $\{u_j\}_{j=1}^{\infty} \subset {^{\pm}}{W}{^{\alpha,p}}(\Omega)$ (or $W^{\alpha,p}(\Omega)$) is bounded, then it is precompact in $L^{p}(\Omega)$.
\end{lemma}

\begin{proof}
We prove the result for $\{u_{j}\}_{j=1}^{\infty} \subset {^{\pm}}{W}{^{\alpha,p}}(\Omega)$. The result for $\{u_{j}\}_{j=1}^{\infty} \subset {W}{^{\alpha,p}}(\Omega)$ follows similarly. 

By assumption, there exists $M >0$ finite so that 
\begin{align}
    \sup_{j} \|u_{j}\|_{{^{\pm}}{W}{^{\alpha,p}}(\Omega)} \leq  M. 
\end{align}
Consider the sequence of mollified functions $\{ u_{j}^{\eps}\}$ and we claim that $u_{j}^{\eps} \rightarrow u_{j}$ in $L^{p}(\Omega)$ uniformly in $j$. See that 
\begin{align*}
    \|u_{j}^{\eps} - u_j \|_{L^{p}(\Omega)} &\leq \|u_{j}^{\eps}\|_{L^{p}(\Omega)} + \| u_{j}\|_{L^{p}(\Omega)} \\ 
    &\leq \|\eta_{\eps}\|_{L^{\infty}(\Omega)} \|u_{j}\|_{L^{1}}(\Omega) + \|u_j\|_{L^{p}(\Omega)} \\ 
    &\leq C \| u_j\|_{L^{p}(\Omega)} \\ 
    &\leq C
\end{align*}
since $u_j$ is a bounded sequence in $L^{p}(\Omega)$. Therefore, $u_{j}^{\eps}\rightarrow u_j$ in $L^{p}(\Omega)$ as $\eps \rightarrow 0$ uniformly in $j$. 

Next, for each fixed $\eps >0$, the sequence $\{u_{j}^{\eps}\}$ is uniformly bounded and equicontinuous. To see this, we estimate for $x \in \Omega$ 
\begin{align*}
    |u_{j}^{\eps}(x)| \leq \|\eta_{\eps}\|_{L^{\infty}(\Omega)} \|u_j\|_{L^{1}(\Omega)} \leq \dfrac{C}{\eps} \|u_j\|_{L^{p}(\Omega)} < \infty
\end{align*}
and 
\begin{align*}
    |{^{\pm}}{\mathcal{D}}{^{\alpha}} u_{j}^{\eps}(x) | &= |\eta_{\eps} * {^{\pm}}{\mathcal{D}}{^{\alpha}} u_j| \\ 
    &\leq \|\eta_{\eps} \|_{L^{\infty}(\Omega)} \|{^{\pm}}{\mathcal{D}}{^{\alpha}} u_j \|_{L^{1}(\Omega)} \\
    &\leq \dfrac{C}{\eps} \|{^{\pm}}{\mathcal{D}}{^{\alpha}} u_j \|_{L^{p}(\Omega)} < \infty.
\end{align*}
Thus $\{u_{j}^{\eps}\}$ is uniformly bounded and these estimates also gives us the equicontinuity. 

Now, fix $\delta >0$. We will show that there exists a subsequence $\{u_{j_m}\} \subset \{u_j\}$ such that $\limsup \|u_{j_m} - u_{j_n}\|_{L^{p}(\Omega)} < \delta$. Select $\eps > 0$ so that 
\begin{align*}
    \|u_{j}^{\eps} - u_{j}\|_{L^{p}(\Omega)} < \dfrac{\delta}{2}
\end{align*}
for any $j$ by the uniformity in $j$. Since $\{u_{j}^{\eps}\}$ is uniformly bounded in $j$ and uniformly equicontinuous in $j$, it follows by Arzela-Ascoli theorem that there exists $\{u_{j_n}^{\eps}\}\subset \{u_j^{\eps}\}$ so that 
\begin{align*}
    \limsup \|u_{j_{m}}^{\eps} - u^{\eps}_{j_{n}}\|_{L^{p}(\Omega)} = 0.
\end{align*}
Then 
\begin{align*}
    &\limsup \|u_{j_{m}} - u_{j_n}\|_{L^{p}(\Omega)}\\
    &\quad \leq \|u_{j_m} - u_{j_m}^{\eps} \|_{L^{p}(\Omega)} + \|u_{j_m}^{\eps} - u_{j_n}^{\eps} \|_{L^{p}(\Omega)} + \|u_{j_n}^{\eps} - u_{j_n} \|_{L^{p}(\Omega)}\\
    &\quad < \delta.
\end{align*}

Finally, for $\delta = 1, 1/2 , 1/3, ...$ via a diagonalization argument, we extract a subsequence $\{u_{j_\ell}\}_{\ell =1}^{\infty} \subset \{u_{j}\}_{j=1}^{\infty}$ satisfying 
\begin{align*}
    \underset{ k, \ell \rightarrow \infty}{\mbox{limsup }} \| u_{j_\ell} - u_{j_ k}\|_{L^{p}(\Omega)} = 0.
\end{align*}
Therefore, $\{u_{j_\ell}\}$ is Cauchy in $L^{p}(\Omega)$. Since this is a Banach space, there exists $u \in L^{p}(\Omega)$ so that $u_{j_\ell} \rightarrow u$ in $L^{p}(\Omega)$. This completes the proof. 
\end{proof}





\end{document}